\documentclass[intlimits,sumlimits,namelimits]{amsart}

\usepackage{amssymb}

\newcommand{\R}{{\mathbb{R}}}

\newcommand{\C}{\mathbb{C}}

\newcommand{\N}{\mathbb{N}}
\newcommand{\Z}{\mathbb{Z}}
\newcommand{\T}{\mathbb{T}}
\newcommand{\D}{\mathbb{D}}

\newcommand{\LongDiscussion}[1]{}

\newtheorem{theorem}{Theorem}[section]
\newtheorem{proposition}[theorem]{Proposition}
\newtheorem{lemma}[theorem]{Lemma}
\newtheorem{corollary}[theorem]{Corollary}

\theoremstyle{definition}

\newtheorem{definition}[theorem]{Definition}

\theoremstyle{remark}

\newtheorem{remark}[theorem]{Remark}

\newcommand{\citet}[1]{\cite{#1}}

\numberwithin{equation}{section}                

\newcommand{\clos}[1]{\overline{#1}}

\newcommand{\Sector}[1]{\Sigma_{#1}}

\renewcommand{\Re}{{\rm Re\,}}
\renewcommand{\Im}{{\rm Im\,}}

\newcommand{\rst}[1]{|{#1}}

\newcommand{\Lscr} {{\mathcal L}}

\newcommand{\sbm}[1]{\left[\begin{smallmatrix} #1
             \end{smallmatrix}\right]}

\makeatletter
\renewcommand{\p@enumii}{}
\makeatother
\let\oldlabel=\label

\renewcommand{\label}[1]{\smash{\raise 10pt\llap{\fbox{\scriptsize
#1}}}\oldlabel{#1}}
\newcommand{\mathlabel}[1]{\smash{\raise 9pt\llap{\scriptsize
(#1)}}\label{#1}}

\renewcommand{\label}[1]{\oldlabel{#1}}
\renewcommand{\mathlabel}[1]{\label{#1}}

\newcommand{\BLO}{\Lscr}

\newcommand{\abs}[1]{\mathopen\vert #1 \mathclose\vert}

\newcommand{\norm}[1]{\mathopen\Vert #1 \mathclose\Vert}

\newcommand{\ConsiderThis}[1]{}

\newcommand{\leaveout}[1]{}

\def\norm#1{\|#1\|}

\title{Microspectral analysis of \\  quasinilpotent operators}

\author{Jarmo Malinen, Olavi Nevanlinna, and Jaroslav Zem\'anek}
\address{Jarmo Malinen,
Department of Mathematics and Systems Analysis,
Aalto University,
P.O.Box 11100,
00076 AALTO, Finland.
{\tt Jarmo.Malinen@aalto.fi.}}

\begin{document}

\begin{abstract}
  We develop a microspectral theory for quasinilpotent linear
  operators $Q$ (i.e., those with $\sigma(Q) = \{ 0 \}$) in a Banach
  space.  When such $Q$ is not compact, normal, or nilpotent, the
  classical spectral theory gives little information, and a somewhat
  deeper structure can be recovered from microspectral sets in
  $\C$. Such sets describe, e.g., semigroup generation, resolvent
  properties, power boundedness as well as Tauberian properties
  associated to $zQ$ for $z \in \C$.
\end{abstract}




\bibliographystyle{amsalpha}

\maketitle

\pagestyle{myheadings}
\thispagestyle{plain}
\markboth{J. MALINEN, O. NEVANLINNA, AND J. ZEM\'ANEK}{MICROSPECTRAL ANALYSIS}



\section{\label{IntroSec} Introduction}

Let $T$ be a bounded linear operator on a complex Banach space $X$
with its spectrum denoted by $\sigma(T)$. Local spectral theory deals
with the local resolvent
\begin{equation*}
  \lambda \mapsto (\lambda - T)^{-1} x, \quad x \in X.
\end{equation*}
The domain and the analytic properties of such functions depend on the
choice of the vector $x$; see \cite{N-R-R:WRLO,F-N-R-R:WRLOII} for
background in this area.  In this paper we discuss related
microspectral questions.

Let $\lambda_0 \in \sigma(T)$ be an isolated point. If $\lambda_0$ is
a \emph{pole} of the global resolvent
\begin{equation*}
  \lambda \mapsto (\lambda - T)^{-1}, \quad \lambda \notin \sigma(T),
\end{equation*}
the the resolvent appears of the same size when approaching the
singularity $\lambda_0$ from all directions; see Proposition
\ref{AlgebraicOperatorProp} below. If, however, $\lambda_0$ is an
\emph{essential singularity}, then the growth depends on the direction
from which the singularity is approached. In order to study this in a
more detailed manner, we proceed as follows:

Let $P_0$ be the Riesz spectral projection of $X$ to the invariant
subspace with respect to $\lambda_0 \in \sigma(T)$, and define $X_0 :=
P X$ and $T_0 = T\rst{X_0}$. Then clearly $Q := T_0 - \lambda_0$ is a
quasinilpotent operator on $X_0$. Rather than studying $(\lambda -
Q)^{-1}$ for $\lambda \neq 0$, we change the variable and consider the
entire function
\begin{equation} \mathlabel{IntroResolvent}
  z \mapsto (I - z Q)^{-1}, \quad z \in \C.
\end{equation}
Refined spectral information for $Q$ can be obtained from the mapping
\begin{equation*}
  j \mapsto (I + z Q)^{j}, \quad j \in \N := \{1, 2, \ldots \}
\end{equation*}
in terms of the set
\begin{equation}
    \mathlabel{PowerBoundedSet} 
    \mathcal B_{Q} := \{ z \in \C: \exists N_z < \infty \text{ such that }
    \norm{(1 + z Q)^{j}} \leq N_z \text{ for all } j \in \N \}; 
\end{equation}
i.e., our requirement is the power-boundedness at the point $z$ with
the bound $N_z$ that possibly depends on $z \in \mathcal B_{Q}$. It is
easy to see that always $0 \in \mathcal B_{Q}$ and that $ \mathcal
B_{Q}$ is convex by the binomial formula for commuting operators.

In addition to the power-bounded set $\mathcal B_{Q}$, a number of
additional sets in $\C$ are defined in \eqref{ResolventCondSet} --
\eqref{TauberianSet} below. Studying these sets is a powerful tool for
understanding the asymptotic behaviour of powers of the families of
operators
\begin{equation} \mathlabel{OperatorFamilies}
 T(z) := I + zQ \quad \text{ and } \quad T_z := (I - zQ)^{-1} 
\end{equation}
for complex $z$. Intuitively speaking, the proposed
\emph{microspectal analysis} amounts to looking at $Q$ from all
possible directions and using all possible magnifying glasses. 

We shall see below in Section \ref{InteriorSec} that topological
properties of a point $z_0$ with respect to the set $\mathcal B_{Q}$
(such as $z_0 \in \mathcal B_Q^{\circ}$, the open interior of
$\mathcal B_Q$) correspond to additional properties of powers
$T(z)^{j}$ (such as growth condition on their consecutive differences
$(I - T(z)) T(z)^{j}$) or the analytic properties of the resolvent
\eqref{IntroResolvent} (like the Ritt resolvent condition).

In particular, it is known that the differences of consecutive powers
$(I - T(z)) T(z)^{j}$ cannot decay arbitrarily fast since either
$\liminf_{j \to \infty}{(j + 1)\norm{(I - T(z)) T(z)^{j}}} \geq 1/e$ or
$Q = 0$; see \cite{JE:QRCIPCBA,MB:IPSAB,K-M-O-T:PBORNE,M-N-Y:LBOXX}.
The question arises whether $z \in \mathcal B_Q$ if the fastest
possible speed of decay is attained, i.e., 
\begin{equation} \mathlabel{TauberianConditionT}
  \sup_{j \geq 1}{(j + 1)\norm{(I - T(z)) T(z)^{j}}} < \infty.  
\end{equation}
Theorem \ref{PhragmenLindelofMainThm} gives an affirmative answer for
\emph{real} operators.

\subsection*{Notation}

The bounded linear operators in a Banach space $X$ are denoted by
$\BLO(X)$. The norm of $X$ and induced operator norm of $\BLO(X)$ are
both denoted by $\norm{\cdot}$. Throughout this paper we assume that
$Q \in \BLO(X)$ with $\sigma(Q) = \{0 \}$.

The natural numbers are $\N := \{1, 2, \ldots \}$.  The complex plane
and the real axis are denoted by $\C$ and $\R$, respectively.  For any
set $A \subset \C$, we denote by $\overline{A}$, $A^c$, $\partial A$,
and $A^\circ$ the closure, complement, boundary, and the (open)
interior of $A$, respectively.  The positive real axis is denoted by
$\R_+ = (0, \infty)$ with $\clos{\R}_+ = [0, \infty)$, and
$\D_{z_0,r} := \{ z \in \C : \abs{ z - x_0} < r \}$. We define the
unit disc $\D := \D_{0,1}$ and its boundary, the unit circle $\T :=
\partial \D$.  If $A,B \subset \C$, we define their product set by $
AB := \{ z s : z \in A \text{ and } s \in B\}$. We say that a set $A
\subset \C$ is \emph{star-like} or that it \emph{consists of full
  rays} if
\begin{equation*}
  A (0,1] \subset A \quad \text{ or }\quad A \R_+ \subset A, 
\end{equation*}
respectively. Note that the set $\{0\}$ satisfies both of these
conditions. 

\emph{Sectors} are convex sets that consist of full rays.  We denote
the \emph{balanced open sectors} in $\C$ by
  \begin{equation} \mathlabel{BalancedOpenSector}
    \Sector{\theta} 
    := \{ r e^{i \theta_0} : 
    r > 0 \text{ and } \theta_0 \in (-\theta, \theta)  \}
\quad \text{ for } \quad  0 < \theta < \pi .
  \end{equation}
  We write $\C_+ := \Sector{\pi/2}$.  General open sectors are the sets
  $e^{i \phi} \Sector{\theta}  $ for $\phi \in [-\pi,\pi)$ and $\theta \in
  (0,\pi)$. The central angle of $e^{i \phi} \Sector{\theta}$ is defined as
  $2 \theta$.  Closed sectors are closures of open sectors or rays
  $e^{i \phi} \clos{\R}_+$ for some $\phi \in [-\pi,\pi)$.

\newpage

\section{\label{BasicSec} Elementary properties}
In addition to the set $\mathcal B_Q$ already introduced in
\eqref{PowerBoundedSet}, we also consider the sets
\begin{align} 
  \mathlabel{ResolventCondSet} \mathcal A_{Q}^k & := \{ z \in \C :
  \limsup_{s \to +\infty}{ \norm{(1 - s z Q)^{-k}}} < \infty \}
  \quad \text{ for } \quad k \in \N,   \\
  \mathlabel{HilleYoshidaSet} \mathcal G_{Q} & := \{ z \in \C :
  \limsup_{t \to +\infty}{ \norm{e^{t z Q}}} < \infty\}, \\
  \mathlabel{KreissSet} \mathcal K_{Q} & := \{ z \in \C:
  \sup_{\Re s > -1/2}{(\abs{s + 1} - \abs{s})\norm{(1 - szQ)^{-1}}} < \infty  \} , \\
  \mathlabel{IteratedKreissSet} \mathcal K_{Q}^\infty & := \{ z \in \C:
  \sup_{\Re s > -1/2, k \in \N}{(\abs{s + 1} - \abs{s})^k\norm{(1 - szQ)^{-k}}} < \infty  \} , \\
  \mathlabel{RittSet} \mathcal R_{Q} & := \{ z \in \C: \sup_{\Re s >
    -1/2}{\norm{(1 - szQ)^{-1}}} < \infty \},
  \quad \text{ and } \\
  \mathlabel{TauberianSet} \mathcal T_{Q}^\alpha & := \{ z \in \C:
  \limsup_{j \to \infty}{(j+1)^\alpha \norm{zQ (1 + z Q)^j}} < \infty
  \} \quad \text{ for } \quad \alpha > 0
\end{align}
as well as the constants defined by
\begin{align*}
  M_z & := \sup_{j \geq 0}{(j + 1) \norm{z Q (1 + z Q)^j}}
  \quad \text{ for } \quad z \in \mathcal T_Q,  \quad \text{ and } \\
  N_z & := \sup_{j \geq 0}{\norm{(1 + z Q)^j}} \quad \text{ for }
  \quad z \in \mathcal B_Q.
\end{align*}
We shall abbreviate $\mathcal T_{Q}^1 = \mathcal T_{Q}$ and $\mathcal
A_{Q}^1 = \mathcal A_{Q}$. If $Q = 0$, then all these sets coincide
with $\C$.

The set $\mathcal A_{Q}^k$ is referred to as the \emph{Abel set} of
order $k$ for obvious reasons.  The \emph{Kreiss set} is so defined
that $z \in \mathcal K_Q$ if and only if $T(z)$ in
\eqref{OperatorFamilies} satisfies the Kreiss resolvent condition
$\norm{(\lambda - T(z))^{-1}} \leq M/(\abs{\lambda} - 1)$ for all
$\abs{\lambda} > 1$. Similarly, the set $\mathcal K_Q^\infty$ relates
to the iterated Kreiss condition $\norm{(\lambda - T(z))^{-k}} \leq
M/(\abs{\lambda} - 1)^k$ for $k \in \N$. For the \emph{Ritt set} we
have $z \in \mathcal R_Q$ if and only if $\norm{(\lambda - T(z))^{-1}}
\leq M/(\abs{\lambda - 1})$ for all $\abs{\lambda} > 1$. Out of the
\emph{Tauberian sets} $\mathcal T_{Q}^\alpha$, only the cases $\alpha
= 1/2$ and $\alpha = 1$ will be studied in this paper, and they
correspond to the differences of consecutive powers of $T(z)$.

Before going any further, let us give examples:
\begin{remark}
  Suppose $Q = \left [ \begin{smallmatrix} 0 & 1 \\ 0 & 0
    \end{smallmatrix} \right ]$.  We get $(1 + zQ)^{k} = \left [
    \begin{smallmatrix} 1 & k z \\ 0 & 1
  \end{smallmatrix} \right ]$, $(1 - zQ)^{-k} = (1 + zQ)^k$, and $e^{z Q} = \left [
    \begin{smallmatrix} 1 & z \\ 0 & 1
  \end{smallmatrix} \right ]$.
It follows $\mathcal A_Q^k = \mathcal B_{Q} = \mathcal G_{Q} =  \{ 0
\}$ for any $k$.  Moreover, $z Q (1 + zQ)^{k} = \left [
    \begin{smallmatrix} 0 &  z \\ 0 & 0
  \end{smallmatrix} \right ]$, and thus $\mathcal T_{Q} = \mathcal T^{1/2}_{Q} = \{ 0 \}$.
\end{remark}
The higher dimensional Jordan matrices and even all \emph{algebraic}
(quasi)nilpotent operators have exactly the same properties:
\begin{proposition} \label{AlgebraicOperatorProp}
   Let $Q \in \BLO(X)$, $Q \neq 0$, be a nilpotent operator. Then
   $\mathcal A_Q = \mathcal B_Q = \mathcal G_Q = \mathcal T^{1/2}_{Q}
   = \{ 0 \}$.
\end{proposition}
\begin{proof}
  If $Q^{n + 1} = 0$, $Q^n \neq0$, and $j > n + 1$, we have $(I -
  T(z))T(z)^j = -zQ (1 + zQ)^j = - \sum_{k = 1}^{n}{\left
    ( \begin{smallmatrix} j \\ k - 1 \end{smallmatrix} \right) z^{k}
    Q^{k} }$.  Now
\begin{equation*}
\lim_{j \to \infty}{\frac{\left ( \begin{smallmatrix} j \\ n - 1 \end{smallmatrix} \right)}
{\sum_{k = 1}^n{\left ( \begin{smallmatrix} j
      \\ k - 1 \end{smallmatrix} \right)}}} = 1, \quad \text{ and hence } \quad 
\lim_{j \to \infty}{\frac{ \sum_{k = 1}^n{\left ( \begin{smallmatrix} j
      \\ k - 1 \end{smallmatrix} \right) z^k Q^k} }
{\sum_{k = 1}^n{\left ( \begin{smallmatrix} j
      \\ k - 1 \end{smallmatrix} \right)}}} = z^n Q^n
\end{equation*}
by convex combinations. We conclude from this that
\begin{equation*}
\lim_{j \to \infty}{\frac{ (I - T(z))T(z)^j x}{ \sum_{k = 1}^n{\left ( \begin{smallmatrix} j
      \\ k - 1 \end{smallmatrix} \right)}}} =  z^n Q^n x \quad \text{ for all } \quad x \in X
\end{equation*}
which gives the estimate 
\begin{equation} \mathlabel{AlgebraicOperatorPropEq1}
   \norm{ (I - T(z))T(z)^j x} > \frac{1}{2} \sum_{k = 1}^n{\left ( \begin{smallmatrix} j
      \\ k - 1 \end{smallmatrix} \right)} \norm{z^n Q^n x} > \frac{1}{2} 
\left ( \begin{smallmatrix} j
      \\ n - 1 \end{smallmatrix} \right)  \norm{z^n Q^n x}
\end{equation}
for all $j$ large enough.  If $Q^n x \neq 0$ and $z \neq 0$, we
conclude from \eqref{AlgebraicOperatorPropEq1} that $z \notin \mathcal
T_Q^{1/2}$ since $\left ( \begin{smallmatrix} j \\ n -
  1 \end{smallmatrix} \right)$ is a polynomial of degree $n - 1$ in
variable $j$.  We have now proved $\mathcal T_Q^{1/2} = \{ 0 \}$ which
implies $\mathcal B_Q = \{ 0 \}$ by claim \eqref{InclusionThmClaim5}
of Theorem \ref{InclusionThm}.  To prove the remaining claims, it is
sufficient (by the same theorem) to treat $\mathcal A_Q$ in a similar
manner.
\end{proof}
\begin{remark} \label{VolterraRemark}
  Let us describe the sets \eqref{ResolventCondSet} --
  \eqref{TauberianSet} in the case $Q = -V^{\alpha}$ where
\begin{equation*}
  (V^{\alpha}f)(x) := \frac{1}{\Gamma(\alpha)} 
  \int_{0}^x {(x - v)^{\alpha - 1} f(v) \, dv}
\end{equation*}
is the quasinilpotent Riemann--Liouville operator on $L^2(0,1)$ for
$\alpha \in (0,1]$.  When
$\alpha = 1$ we have the Volterra operator that satisfies
\begin{equation*}
  \mathcal B_{Q} = \mathcal G_{Q} = \mathcal T^{1/2}_{Q} =  \clos{\R}_+, \quad
\mathcal A^{k}_{Q} = \C_+, \quad \text{ and }  \quad\mathcal T_{Q} = \mathcal R_{Q} = \{ 0 \}
\end{equation*}
see \cite[Theorem 1]{DT:OPBCVOP}, \cite[Theorem 1.1]{YL:PWRCDCVO} as
well as Theorem \ref{InclusionThm} below. Further examples of quasinilpotent
operators are given by $(V_\phi f)(x) = (V f)(\phi(x))$ for $\phi \in
C[0,1]$, in which case $V_\phi$ is quasinilpotent if and only if
$\phi(x) \leq x$ for all $x \in [0,1]$; see \cite{YT:QIO} and
\cite{RW:SVCO}. For $\alpha \in (0,1)$, we have $\clos{\R}_+
\subset \mathcal R_{Q}$; see \cite[p.  137]{YL:SPSOSRRC} and also
\cite{ND:OIFPO}.  This provides us with an example of a bounded
analytic semigroup generated by a quasinilpotent operator.
\end{remark}


We give next the elementary properties of the sets defined
in \eqref{TauberianSet}--\eqref{HilleYoshidaSet} based on a direct
application of well-known results.

\begin{proposition} \label{TrivialitiesProp}
  Let $Q \in \BLO(X)$ be quasinilpotent. Then the following holds:
  \begin{enumerate}
  \item \label{TrivialitiesPropClaim2}
  The sets $\mathcal B_{Q}$
  and $\mathcal G_{Q}$ are convex.
  \item \label{TrivialitiesPropClaim4} The sets $\mathcal B_{Q}$ and
    $\mathcal K_{Q}^\infty$ are star-like.
  \item \label{TrivialitiesPropClaim1} $\mathcal A_{Q}^k \R_+ =
    \mathcal A_{Q}^k$ for $k \in \N$ and $\mathcal G_{Q} \R_+ =
    \mathcal G_{Q}$; i.e., they consist of full rays.
  \end{enumerate}
\end{proposition}
\begin{proof}
  Since convex combinations of power bounded operators are
  power-bounded, we have $\alpha z_1 + \beta z_2 \in \mathcal B_{Q}$
  if $z_1, z_2 \in \mathcal B_{Q}$ and $\alpha, \beta \geq 0$ satisfy
  $\alpha + \beta = 1$. That $\mathcal B_{Q}$ is star-like follows
  from convexity and the fact that $0 \in \mathcal B_{Q}$.  The full
  ray property of the sets is trivial as well as convexity of
  $\mathcal G_{Q}$.

  It remains to prove that $\mathcal K_{Q}^\infty$ is star-like.
  Using the Hille--Yoshida generator theorem we see that each $e^{i
    \phi} - 1 + e^{i \phi} z Q$ for $\phi \in [-\pi, \pi)$ generates a
    bounded semigroup (with an upper bound $M_z$ not depending on
    $\phi$) if and only if $z \in \mathcal K_Q^\infty$; see
    \cite[p. 248--249]{ON:GROPB}.  For $\alpha \in [0,1]$ and $t \geq
    0$ we have
  \begin{equation} \mathlabel{TrivialitiesPropEq2}
    \norm{e^{t ( e^{i \phi} - 1 + e^{i \phi} \alpha z Q )}} \leq
    e^{-t(1 - \alpha)(1 - \cos{\phi})}     \norm{e^{t ( e^{i \phi} - 1 + e^{i \phi} z Q )}} \leq M_z
  \end{equation}
  and the claim follows.
\end{proof}
\begin{remark}
\label{ZeroBoundaryPointProp}
  We conclude that  $\mathcal G_{Q}$ is a convex sets consisting of full rays,
  i.e., a sector, a single ray, or just the set $\{ 0 \}$. Thus, either $\mathcal
  G_{Q} \subset e^{i \phi} \clos{\C}_+$ for some $\phi \in [-\pi, \pi)$ or
    $\mathcal G_{Q} = \C$ which implies the boundedness of the entire
    function $e^{tQ}$; hence $Q = 0$ by the Liouville's theorem.
\end{remark}

The Hille--Yoshida generator theorem for continuous semigroups takes
the following form:
\begin{proposition} \label{HilleYosidaProp}
  Defining the operators $T_z$ for $z \in \C$ by
  \eqref{OperatorFamilies}, we have
  \begin{equation} \mathlabel{HilleYoshidaChar}
    \mathcal G_Q  = \{ z \in \C :
    \sup_{s > 0,k \in \N} {\norm{T_{sz}^k}}  < \infty  \}
     = \{ z \in \C : \limsup_{s \to +\infty} \left ( \sup_{k \in \N}
      {\norm{T_{sz}^k}} \right ) < \infty \}.
  \end{equation}
\end{proposition}
This Gelfand--Hille theorem is a consequence of Proposition
\ref{HilleYosidaProp}:
\begin{proposition} \label{GelfandHilleProp}
  Let $Q \in \BLO(X)$ be a quasinilpotent operator. Suppose there is
  $z \in \mathcal G_Q$, $z \neq 0$, such that $-z \in \mathcal
  T_Q^{1/2}$. Then $Q = 0$.
\end{proposition}
\noindent The conclusion can be written as $ (- \mathcal G_Q) \cap
\mathcal T_Q^{1/2} = \{ 0 \}$ for $Q \neq 0$.
\begin{proof}
We clearly have
\begin{equation*}
  \norm{I - T(-z)}  \leq \norm{T(-z)^{-j}} \cdot \norm{(I - T(-z)) T(-z)^j}.
\end{equation*}
Now $T(-z)^{-1} = (I  -z Q)^{-1} = T_{z}$. If $z \in \mathcal G_Q$,
then $\norm{T_{z}^j} \leq M < \infty$ for all $j > 0$.  If $-z \in
\mathcal T_Q^{1/2}$, we have $\norm{(I - T(-z)) T(-z)^j} \leq
\frac{C}{\sqrt{j + 1}}$ for all $j > 0$. Then 
\begin{equation*}
  \norm{I - T(-z)}  \leq \frac{ M C}{\sqrt{j + 1}} \to 0 \quad \text{ as } j \to \infty; 
\end{equation*}
hence $T(-z) = I$, and $Q = 0$ follows if $z \neq 0$.
\end{proof}
We remark that the classical Gefand--Hille theorem (see \cite[Theorem
  1]{JZ:OGHT})) implies that neither of the sets $\mathcal B_Q$ or
$\mathcal G_Q$ contains a line $e^{i \phi} \R$ for $\phi \in
[-\pi/2, \pi/2)$ unless $Q = 0$.  Indeed, if $\R \subset \mathcal G_Q$
  then $T := e^Q$ would be an operator with $\sigma(T) = \{ 1 \}$ and
  $\sup_{k \in \Z}{\norm{T^k}} < \infty$. The claim about $\mathcal
  B_Q$ follows from the inclusion $\mathcal B_Q \subset \mathcal G_Q$
  given in Theorem \ref{InclusionThm} below.
\begin{proposition} \label{NevanlinnaLubichProp} 
  The following are equivalent for $z = re^{i \theta} \in \C$:
  \begin{enumerate}
  \item \label{NevanlinnaLubichPropCondition2} There exists a constant
    $C < \infty$ such that
    \begin{equation*}
      \norm{(1 - sz  Q)^{-1}} \leq C \quad \text{  for all } s \text{ with }
      \Re s > -1 / 2.
    \end{equation*}
    (i.e., $z \in \mathcal R_Q$) 
  \item \label{NevanlinnaLubichPropCondition3a} There exists $\delta >
    0$ with the following property: For any $\eta \in [0,\delta)$,
    there exists a constant $C_{\eta} < \infty$ such that
    \begin{equation} \mathlabel{ExtendedRitta}
      \norm{(1 - s z Q)^{-1}} \leq C_{\eta}
      \quad \text{ for all } \quad  s \in \Sector{\pi / 2 +
        \eta} \cup \{ \Re{s} > -1/2\}.
    \end{equation}
  \item \label{NevanlinnaLubichPropCondition3} There exists $\delta >
    0$ with the following property: For any $\eta \in [0,\delta)$,
    there exists a constant $C_{\eta} < \infty$ such that
    \begin{equation} \mathlabel{ExtendedRitt}
      \norm{(1 - sQ)^{-1}} \leq C_{\eta}
      \quad \text{ for all } \quad  s \in e^{i \theta} \Sector{\pi / 2 +
        \eta} \cup \{ \Re{(e^{-i \theta} s)} > -r/2\}.
    \end{equation}
  \item \label{NevanlinnaLubichPropCondition1} $T(z) = I + z Q$ is
    power-bounded, and it satisfies the Tauberian condition $\sup_{n
      \in \N}{(n+1) \norm{\left (I - T(z) \right ) T(z)^n}} <
    \infty$.
  \end{enumerate}
\end{proposition}
\noindent   This is plainly the characterisation of Ritt's operators in
  \cite[Proposition 1.1]{M-N-Y:OTCFBLO} when applied to $T(z)$ in
  \eqref{OperatorFamilies}. 
  
There is a number of inclusions that the sets in
\eqref{ResolventCondSet} -- \eqref{TauberianSet} satisfy:
\begin{theorem} \mathlabel{InclusionThm}
Let  $Q \in \BLO(X)$ be quasinilpotent. Then the following holds:
\begin{enumerate}
\item \label{InclusionThmClaim2} $\mathcal R_Q = \mathcal B_Q
  \cap \mathcal T_Q $;
\item \label{InclusionThmClaim4} $\mathcal G_Q \C_+ \subset
  \mathcal A_Q^k$ for all $k \in \N$ and, in particular, $\mathcal R_Q
  \C_+ \subset \mathcal A_Q$;
\item \label{InclusionThmClaim1} $\mathcal B_{Q} \subset \mathcal
  B_{Q} \R_+ \subset \mathcal G_{Q} \subset \mathcal A_Q^k$ for all $k
  \in \N$;
\item \label{InclusionThmClaim2b} $\mathcal R_Q = \mathcal A_Q^k
  \cap \mathcal T_Q $  for all $k \in \N$;
\item \label{InclusionThmClaim3} $\mathcal T_Q \subset \mathcal
  T_Q^{1/2}$ and $\mathcal G_Q \cap \mathcal T_Q^{1/2}   \subset
  \mathcal B_Q$;
\item \label{InclusionThmClaim6} $\mathcal R_Q = \mathcal G_Q
  \cap \mathcal T_Q$;
\item \label{InclusionThmClaim7} $\mathcal R_Q \R_+ = \mathcal
  R_Q$;
\item \label{InclusionThmClaim5} $\mathcal B_Q [0,1) \subset
  \mathcal T_Q^{1/2}$;
\item \label{InclusionThmClaim8} $\mathcal B_{Q} \subset \mathcal
  K_{Q}$ and $\mathcal K_{Q} \C_+ \subset \mathcal A_Q$;
\item \label{InclusionThmClaim8b} $\mathcal B_{Q} \subset \mathcal
  K_{Q}^\infty \subset \mathcal G_Q \cap \mathcal K_{Q}$; and
\item \label{InclusionThmClaim9}  
  $\left(- \mathcal B_{Q^2} \right )^{1/2} \cap \mathcal G_Q \subset \mathcal B_Q$.
\end{enumerate}
\end{theorem}
\begin{proof}
  Claim \eqref{InclusionThmClaim2} and the latter part of
  \eqref{InclusionThmClaim4} follow directly from Proposition
  \ref{NevanlinnaLubichProp}. The first inclusion of
  \eqref{InclusionThmClaim4} follows from the Hille--Yosida theorem:
  if $z \in \mathcal G_Q$, then for all $s \in \C_+$and $k \in \N$ we
  have $\norm{(s^{-1} - z Q )^{-k}} \leq M/(\Re{s^{-1}})^k$; i.e.,
  $\norm{(I - s z Q )^{-k}} \leq M /cos^k{\phi}$ for some constant $M <
  \infty$ where $\cos{\phi} = \Re{s}/\abs{s}$.  Claim
  \eqref{InclusionThmClaim1} follows from the fact that the power
  series coefficients of $e^x$ are positive, together with
  \eqref{HilleYoshidaChar} and the fact that $\mathcal G_{Q} \R_+ =
  \mathcal G_{Q}$ by Proposition \ref{TrivialitiesProp}.
  Let us prove the first part of claim \eqref{InclusionThmClaim2b}.
  The inclusion $\mathcal A_Q \cap \mathcal T_Q \subset \mathcal B_Q$
  follows by applying \cite[Theorem 1]{M-N-Y:OTCFBLO} to operators $1
  + z Q$. Hence $\mathcal A_Q \cap \mathcal T_Q \subset \mathcal B_Q
  \cap \mathcal T_Q = \mathcal R_Q$ by claim
  \eqref{InclusionThmClaim2}. The converse inclusion follows from
  claims \eqref{InclusionThmClaim2} and \eqref{InclusionThmClaim1}.
    
  Claim \eqref{InclusionThmClaim3} is another tauberian theorem, and
  it follows from \cite[Theorem III.5 on p. 68]{AP:LOS} as pointed out
  in \cite{M-N-Y:OTCFBLO}.

  Claim \eqref{InclusionThmClaim7} is given in \cite[Proposition
    2]{DT:OPBCVOP} but we prove it here, too.  Let $z = re^{i
    \theta}\in \mathcal R_Q$ and $h > 0$ be given.  Then there exists
  $\delta > 0$ such that $Q$ satisfies condition
  \eqref{NevanlinnaLubichPropCondition3} of Proposition
  \ref{NevanlinnaLubichProp}. This is equivalent with having
  $\norm{(1 - s\cdot hQ)^{-1}} \leq C_{\theta,\delta}$ for any $\eta
  \in [0,\delta)$ and all $s \in e^{i \theta} \Sector{\pi / 2 + \eta} \cup
  \{ \Re{(e^{-i \theta} s)} > -r/2h\}$ since $h^{-1}e^{i \theta}
  \Sector{\pi / 2 + \eta} = e^{i \theta} \Sector{\pi / 2 + \eta}$ and $h^{-1}\{
  \Re{(e^{-i \theta} s)} > -r/2\} = \{ \Re{(e^{-i \theta} s)} >
  -r/2h\}$. If $h > 1$, we do not have the inclusion $\{ \Re{(e^{-i
      \theta} s)} > -r/2\} \subset \{ \Re{(e^{-i \theta} s)} >
  -r/2h\}$ but it is nevertheless easy to see that the estimate
  $\norm{(1 - s\cdot hQ)^{-1}} \leq C'_{\theta,\delta}$ holds (with a
  larger constant $C'_{\theta,\delta} < \infty$ in place of
  $C_{\theta,\delta}$) even for $s$ in the larger set $e^{i \theta}
  \Sector{\pi / 2 + \eta} \cup \{ \Re{(e^{-i \theta} s)} > -r/2\}$ since
  this set differs from $e^{i \theta} \Sector{\pi / 2 + \eta} \cup \{
  \Re{(e^{-i \theta} s)} > -r/2h\}$ only by a precompact set.
  Proposition \ref{NevanlinnaLubichProp} proves now claim
  \eqref{InclusionThmClaim2}.
  
  Claim \eqref{InclusionThmClaim5} follows from the fact that for
  $\alpha \in [0,1)$ and power-bounded $T$, the bound $\sqrt{j + 1} \,
  \norm{(I - T_\alpha)T_\alpha^j} \leq M_\alpha < \infty$ holds for
  all operators $T_\alpha = \alpha + (1 - \alpha)T$, see \cite[Theorem
  4.5.3]{ON:CIL}. See also \cite[Remark 2]{DT:OPBCVOP} and \cite[Lemma
  2.1]{F-W:CPSCMO}.
  
  Let us prove claim \eqref{InclusionThmClaim8}.  The inclusion
  $\mathcal B_{Q} \subset \mathcal K_{Q}$ follows because
  power-boundedness implies (even the iterated) the Kreiss resolvent
  condition by a straightforward argument.  Let us prove the latter
  inclusion. Clearly for all $s \in \C$ we have $\abs{s + 1} - \abs{s}
  = \frac{2 \Re{s} + 1}{\abs{s + 1} + \abs{s}}$, and $\abs{s} < \abs{s
    + 1}$ for $s \in \C_+$.  Thus for $s \in \C_+$ we have
  $\cos{\theta} - \frac{1}{2\abs{s + 1}} < \abs{s + 1} - \abs{s} <
  \cos{\phi} + \frac{1}{2\abs{s}}$ where $\cos{\theta} = \frac{\Re{(s
      + 1)}}{\abs{s + 1}} > \frac{\Re{s}}{\abs{s}} = \cos{\phi}$ and
  $0 \leq \phi < \theta$.  It follows that for $\phi \in [0, \pi/2)$,
    there are costants $m_\phi, M_\phi$ such that
  \begin{equation} \mathlabel{TrivialitiesPropEq1}
    0 < m_\phi \leq \abs{s + 1} - \abs{s} < M_\phi < \infty
  \end{equation}
 for  all $s \in \Sector{\phi}$. 
 
 Now, let $z \in \mathcal K_Q$ and $\phi \in [0, \pi/2)$. Then for
 any $\alpha \in [-\phi, \phi]$ we get
 \begin{equation*}
   \sup_{r \geq 0}{\norm{(1 - r \cdot e^{i \alpha} zQ)^{-1}}}
   \leq \sup_{s \in \Sector{\phi}}{\norm{(1 - s zQ)^{-1}}} < \infty
 \end{equation*}
 by equation \eqref{TrivialitiesPropEq1}. Thus $e^{i \alpha} z \in
 \mathcal A_Q$ and also $\Sector{\phi} z \subset \mathcal A_Q$.
 Because $z \in \mathcal K_Q$ and $\phi \in (-\pi/2, \pi/2)$ are
 arbitrary, we conclude that $\mathcal K_{Q} \C_+ \subset \mathcal
 A_Q$.  

The nontrivial part $\mathcal K_Q^\infty \subset \mathcal G_Q$ of
claim \eqref{InclusionThmClaim8b} follows from the estimate
$\norm{e^{t e^{i \phi} zQ}} \leq M_z e^{t(1 - \cos{\phi})}$ for all
$\phi \in [-\pi, \pi)$ that can be proved in a similar way as
  \eqref{TrivialitiesPropEq2}.
 
 Finally claim \eqref{InclusionThmClaim9}: Note that
 \eqref{OperatorFamilies} implies $T(z)^k = (1 - z^2 Q^2)^k T_z^k$ for
 all $k \in \N$. If $z^2 \in -\mathcal B_{Q^2}$, then $1 - z^2 Q^2$ is
 power-bounded.  Since $z \in \mathcal G_Q$, the operator $T_z$ is
 power-bounded, and thus $z \in \mathcal B_Q$.
\end{proof}
\begin{remark}
  Claim \eqref{InclusionThmClaim2} is strengthened in Theorem
  \ref{InteriorBQoThm}.  It is instructive to compare claims
  \eqref{InclusionThmClaim4} and \eqref{InclusionThmClaim8}.  By claim
  \eqref{InclusionThmClaim5} we get $\mathcal B_Q^\circ \subset
  \mathcal T_Q^{1/2}$ where $\mathcal B_Q^\circ$ denotes the open
  interior of $\mathcal B_Q$. A stronger form of this claim is given
  by Proposition \ref{InterioInclusionProp}.
\end{remark}
\begin{remark}
  Note that \eqref{InclusionThmClaim4} and \eqref{InclusionThmClaim2b}
  imply that the inclusion $\mathcal R_Q = \mathcal A_Q \cap \mathcal
  T_Q \subset \mathcal A_Q$ is strict if $\mathcal R_Q$ has a nonempty
  interior.  Hence the inclusion $\mathcal A_Q \subset \mathcal T_Q$
  does not generally hold, and even $\mathcal A_Q \subset \mathcal
  T_Q^{1/2}$ fails by the remark following Theorem
  \ref{PhragmenLindelofMainThm}.  It is an open question whether the
  converse inclusion $\mathcal T_Q^{1/2} \subset \mathcal A_Q$ always holds.
  This and the stronger inclusion $\mathcal T_Q \subset \mathcal B_Q$
  are given in Theorem \ref{PhragmenLindelofMainThm} under quite
  restrictive assumptions.
\end{remark}

\begin{corollary} \label{InclusionThmCor} 
For a quasinilpotent $Q \in \BLO(X)$ we have $\mathcal B_Q = \mathcal
G_Q$ if and only if $\mathcal G_Q \subset \mathcal T_Q^{1/2}$.
\end{corollary}
\noindent Indeed, this follows from claims \eqref{InclusionThmClaim1}
and \eqref{InclusionThmClaim3} of Theorem \ref{InclusionThm}.

A slight improvement of Proposition \ref{NevanlinnaLubichProp}   is possible:
\begin{proposition} \label{RittSetProp}
  For any quasinilpotent $Q \in \BLO(X)$ we have
\begin{equation} \mathlabel{RittSetImproved}
 \mathcal R_{Q} 
 = \{ z \in \C: \sup_{\Re s > 0 }{\norm{(1 - szQ)^{-1}}} < \infty \}.  
\end{equation}
\end{proposition}
\begin{proof}
  It is clear by continuity that $\sup_{\Re s > 0 }{\norm{(1 -
      szQ)^{-1}}} = \sup_{\Re s \geq 0 }{\norm{(1 - szQ)^{-1}}}$.  By
  \eqref{RittSet}, it is thus enough to show that 
  \begin{equation*}
 A := \{ z \in \C:
  \sup_{\Re s \geq 0 }{\norm{(1 - szQ)^{-1}}} < \infty \} \subset
   \mathcal R_{Q}.   
  \end{equation*}
  Fix $z \in A$, and define $M := \sup_{\Re s \geq 0 }{\norm{\left (I
        - s z Q \right )^{-1}}} < \infty$ and $\alpha := \frac{1}{2 M
    \norm{z Q}} > 0$.  Take any $s = x + yi \in \C$ where $x \in (-\alpha,
  0]$ and $y \in \R$. Now
  \begin{equation*}
    \norm{(I - s z Q)^{-1}} \leq \norm{(I - iy z Q)^{-1}} \cdot
    \norm{\left (I - (I - iy z Q)^{-1} (s - i y) z Q \right )^{-1}}
  \end{equation*}
  where $\norm{(I - iy z Q)^{-1}} < M$ and
  \begin{equation*}
   \norm{ (I - iy z Q)^{-1} (s - i y ) z Q } \leq M \alpha \norm{z Q} = \frac{1}{2}.
  \end{equation*}
  It follows from this that
  \begin{equation} \mathlabel{RittSetPropEq1}
    \norm{\left (I - s z Q \right )^{-1}} \leq M_\alpha \quad \text{ for all }
    s \text{ satisfying }  - \alpha < \Re{s} \leq 0
  \end{equation}
  holds with $M_\alpha = 2M$, and hence for all $s$ with $\Re{s} > -
  \alpha$ because $z \in A$.  Condition \eqref{RittSetPropEq1} for
  $\Re{s} > - \alpha$ is clearly equivalent with
  \begin{equation*} 
    \norm{\left (I - \tfrac{s}{2 \alpha} \cdot 2 \alpha z Q \right )^{-1}} < M_\alpha \quad \text{ for all }
    s \text{ satisfying } \Re{\tfrac{s}{2 \alpha}} > - 1/2.
  \end{equation*}
  which by \eqref{RittSet} is equivalent with $2 \alpha z \in \mathcal
  R_Q$, and further equivalent with $z \in \mathcal R_Q$ by claim
  \eqref{InclusionThmClaim7} of Theorem \ref{InclusionThm}.
 \end{proof} 
Clearly \eqref{RittSetImproved} implies trivially $\mathcal R_Q \R_+
\subset \mathcal R_Q$ but we use it in the above proof.  Using the
sectorial extension property in claim
\eqref{NevanlinnaLubichPropCondition3a} of Proposition
\ref{NevanlinnaLubichProp} and the boundedness of the resolvent in any
compact set, we see that whenever $z \in \mathcal R_Q$ holds, we have
   \begin{equation*}
     \sup_{\Re{s} > \alpha}{\norm{\left (I
         - s z Q \right )^{-1}}} < \infty \quad \text{ for all } \quad \alpha < 0. 
   \end{equation*}
  The resolvent estimation technique in the proof of Proposition
  \ref{RittSetProp} shows also the following: If $M_\alpha :=
  \sup_{\Re{s} > \alpha}{\norm{\left (I - s z Q \right )^{-1}}} <
  \infty$ for \emph{some} $\alpha > 0$, then $M_\beta < \infty$ holds
  for some $\beta \in (0,\alpha)$. Indeed, for any $\alpha > 0$ it
  follows from $\sup_{\Re{s} > \alpha}{\norm{\left (I - s z Q \right
      )^{-1}}} < \infty$ by continuity that $\sup_{\Re{s} \geq
    \alpha}{\norm{\left (I - s z Q \right )^{-1}}} < \infty$. This
  implies $\sup_{\Re{s} > \beta}{\norm{\left (I - s z Q \right
      )^{-1}}} < \infty$ for some $\beta \in (0,\alpha)$ by an
  estimation using the resolvent identity.  However, we cannot exclude
  the possibilities that $\gamma := \inf\{\beta : M_\beta < \infty\}
  > 0$ or $M_\gamma = \infty$ even in the case when $\gamma = 0$.

An interpolation between $\mathcal R_Q$ and $\mathcal T_Q^{1/2}$
produces the following result:
\begin{proposition} \label{RittWeakTauberianInterpProp}
  For a quasinilpotent $Q \in \BLO(X)$ we have
  \begin{equation*}
    \left \{ \alpha z_1 + (1 - \alpha) z_2 : \alpha \in [0,1], z_1 \in \mathcal T_Q^{1/2}, z_2 \in \mathcal R_Q \right \}
 \subset \mathcal T_Q^{1/2}.
  \end{equation*}
\end{proposition}
\noindent Since always $0 \in \mathcal R_Q$, we conclude that $\mathcal T_Q^{1/2}$ is a star-like set.
\begin{proof}
Fix $z_1 \in T_Q^{1/2}$, $z_2 \in \mathcal R_Q$, and $\alpha \in
[0,1]$. Then there are constants $0 < M_1, M_2, M_3, M_4 < \infty$ such that the
estimates
\begin{equation}
\mathlabel{RittWeakTauberianInterpPropEq1}
\begin{aligned}
& \norm{z_1 Q T(z_1)^j} \leq \frac{M_1}{\sqrt{j + 1}}, \quad \norm{T(z_1)^{j} } \leq M_2 \sqrt{j + 1} \\
&  \norm{z_2 Q T(z_2)^{k - j}} \leq \frac{M_3}{k - j + 1}, \text{ and } \norm{T(z_2)^{k - j}} \leq M_4
\end{aligned}
\end{equation}
hold for all $j,k$ satisfying $0 \leq j \leq k$. For the
second estimate, it is enough to estimate the sum $I - T(z_1)^j =
\sum_{l = 0}^{j - 1}{(I - T(z_1))T(z_1)^l}$ using the first inequality
in \eqref{RittWeakTauberianInterpPropEq1} to obtain $\norm{T(z_1)^j}
\leq 1 + 2 M_1 \left (\sqrt{j + 1} - 1 \right )$.  The latter two
estimates follow from claim \eqref{InclusionThmClaim2} of Theorem
\ref{InclusionThm}.

Now, define $\tilde T(\alpha) := T(\alpha z_1 + (1 - \alpha) z_2 ) =
\alpha T(z_1) + (1 - \alpha) T(z_2)$ for $\alpha \in [0,1]$.  For
all $k$ we have
\begin{align*}
  -\left (I - \tilde T(\alpha) \right ) \tilde T(\alpha)^k & = 
\sum_{j = 0}^{k}{ \left ( \begin{matrix} k \\ j \end{matrix}\right ) \alpha^{j + 1} [z_1 Q T(z_1)^j] (1 - \alpha)^{k - j} T(z_2)^{k - j} } \\
& +  \sum_{j = 0}^{k}{ \left ( \begin{matrix} k \\ j \end{matrix}\right ) \alpha^{k} T(z_1)^j  (1 - \alpha)^{k - j + 1} [z_2 Q  T(z_2)^{k - j}] }.
\end{align*}
By using the estimates \eqref{RittWeakTauberianInterpPropEq1} we get 
\begin{align*}
 \norm{\left (I - \tilde T(\alpha) \right ) \tilde T(\alpha)^k}  & 
\leq M_1 M_4 
\sum_{j = 0}^{k}{ \frac{1}{\sqrt{j + 1}}  \left ( \begin{matrix} k \\ j \end{matrix}\right ) \alpha^{j + 1} (1 - \alpha)^{k - j} } \\
& + M_2 M_3 
\sum_{j = 0}^{k}{ \frac{\sqrt{j + 1}}{k -j + 1}  \left ( \begin{matrix} k \\ j \end{matrix}\right ) \alpha^{j} (1 - \alpha)^{k - j + 1} }.
\end{align*}
Note that $\sqrt{ \frac{k + 1}{j + 1}} \left ( \begin{matrix} k
  \\ j \end{matrix}\right ) = \sqrt{ \frac{j + 1}{k + 1}} \left
( \begin{matrix} k + 1 \\ j + 1 \end{matrix}\right ) \leq \left
( \begin{matrix} k + 1 \\ j + 1 \end{matrix}\right )$ and
$\frac{\sqrt{(k + 1) (j + 1)}}{k -j + 1} \left ( \begin{matrix} k
  \\ j \end{matrix}\right ) = \linebreak \sqrt{ \frac{j + 1}{k + 1}}
\left ( \begin{matrix} k + 1 \\ j \end{matrix}\right ) \leq \left
( \begin{matrix} k + 1 \\ j \end{matrix}\right )$ for all $0 \leq j
\leq k$. By the binomial theorem, we get $\sqrt{k + 1} \norm{\left (I
  - \tilde T(\alpha) \right ) \tilde T(\alpha)^k} \leq M_1 M_4 (1 - (1
- \alpha)^{k + 1}) + M_2 M_3 (1 - \alpha^{k + 1})$.
\end{proof}
\begin{remark} \label{SquareRootPowerGrowthRemark}
Using the same technique as in the proof of Proposition
\ref{RittWeakTauberianInterpProp}, we get the estimate $\norm{\left (
  I - T(z) \right )T(z)^j} \leq M \ln{(j + 1)}/\sqrt{j + 1}$ where $z
= \alpha z_1 + (1 - \alpha) z_2$ for $\alpha \in [0,1]$, $z_1 \in
\mathcal T_Q$, and $z_2 \in \mathcal B_Q$. Similarly, we have
$\norm{\left ( I - T(z) \right )T(z)^j} \leq M \ln{(j + 1)}/(j + 1)$
if $z_1 \in \mathcal T_Q$, and $z_2 \in \mathcal R_Q$ instead.  Hence,
if $z_1 \in \mathcal T_Q \setminus \mathcal A_Q$ and $z_2 \in \mathcal
R_Q \subset \mathcal A_Q$ we have at most
logarithmic growth in the powers $T(z)^j$ for those $z = \alpha z_1 +
(1 - \alpha) z_2$, $\alpha \in (0,1]$, that satisfy $z \in \mathcal
  A_Q$; this follow by modifying the proof of Tauberian theorem
  \cite[Theorem 1]{M-N-Y:OTCFBLO}.
 
 Note that in the proof of Proposition
 \ref{RittWeakTauberianInterpProp}, we use for $z \in \mathcal
 T_Q^{1/2}$ the estimate $\norm{T(z)^j} \leq M \sqrt{j + 1}$ which
 cannot be improved by the Tauberian approach used in \cite[Theorem
   1]{M-N-Y:OTCFBLO} even if $z \in \mathcal A_Q \cap \mathcal
 T_Q^{1/2}$ .
\end{remark}

\section{\label{InteriorSec} Interior points of $\mathcal B_Q$ and $\mathcal R_Q$}

Recall that $\mathcal B_Q^\circ$ denotes the open interior of
$\mathcal B_Q$.
\begin{proposition} \mathlabel{MaximumPropertyProp}
  The function $\mathcal B_Q^\circ \ni z \mapsto N_z \in \R_+$ defined
  by
  \begin{equation} \mathlabel{MaxPowerFunction}
    N_z := \sup_{n \geq 1}{\norm{(1 + zQ)^n}}
  \end{equation}
  is uniformly bounded on compact subsets of $\mathcal B_Q^\circ$.
\end{proposition}
\begin{proof}
  Let $z_0 \in \mathcal B_Q^\circ$ be arbitrary, and define $\delta =
  \tfrac{1}{2} \mathop{dist}{(z_0, \partial \mathcal B_Q)}$. Define $D
  = \{ z \in \C : \abs{z - z_0} < \delta \}$.
  
  Then there exists a regular, convex polygon around $D$ of, say, $n$
  vertices $v_1, v_2, \ldots , v_n$ such that
  \begin{equation*}
    D \subset \mathop{conv}\{ v_1, v_2, \ldots , v_n  \} 
    \subset \mathcal B_Q^\circ
  \end{equation*}
  where $\mathop{conv}$ denotes the closed, convex hull. Now, the set
  $\mathop{conv}\{ v_1, v_2, \ldots , v_m \}$ can be written as a
  union of $m - 2$ closed triangles whos vertices are in the set $\{
  v_1, v_2, \ldots , v_m \}$. To show that $z \mapsto N_z$ is
  uniformly bounded on $D$, it is enough to show that the same
  function is bounded on all closed triangles inside $\mathcal
  B_Q^\circ$.
  
  So, let $v_1, v_2, v_3$ be the verteces of a triangle, and define $M
  = \max{(M_{v_1}, M_{v_2}, M_{v_3})}$ (see \eqref{MaxPowerFunction}).
  Recall the trinomial formula $(a + b + c)^n = \sum_{j,k = 0}^n
  {c^n_{j,k} a^j b^k c^{n - j - k}}$ where the coefficients
  $c^n_{j,k}$ are natural numbers. We have for any $\alpha, \beta,
  \gamma \in [0,1]$, $\alpha + \beta + \gamma = 1$, and $n \in \N$ the
  estimate
  \begin{align*}
    & \norm{\left ( 1 + (\alpha v_1 + \beta v_2 + \gamma v_3) Q \right
      )^n} = \norm{\left ( \alpha (1 + v_1 Q) + \beta (1 + v_2
        Q) + \gamma (1 + v_3 Q) \right )^n} \\
    & \leq \sum_{j,k = 0}^n {c^n_{j,k} \alpha^j \beta^k \gamma^{n - j
        - k} \norm{(1 + v_1 Q)^j} \cdot \norm{(1 + v_2 Q)^k} \cdot
      \norm{(1 + v_3 Q)^{n - j - k}} } \\
    & \leq M^3 (\alpha + \beta + \gamma)^n = M^3.
  \end{align*}
  Thus $\sup_{z \in \mathop{conv}(v_1, v_2, v_3)}{N_z} \leq M^3 <
  \infty$.  We have now proved that $z \mapsto N_z$ is uniformly
  bounded on any disk $D$ whose closure is in $\mathcal
  B_Q^\circ$. The proof is completed by a usual covering argument.
\end{proof} 


\begin{proposition} \label{InterioInclusionProp}
  Always $\mathcal B_Q^\circ \subset \mathcal T_Q$ for any
  quasinilpotent $Q \in \BLO(X)$.
\end{proposition}
\begin{proof}
  By the Cauchy integral, we have
  \begin{equation*}
    f'(z_0) = \frac{1}{2\pi i} \int_\Gamma{ f'(\xi)(\xi - z_0)^{-1} d\xi }
= \frac{1}{2\pi i} \int_\Gamma{ f(\xi)(\xi - z_0)^{-2} d\xi }
  \end{equation*}
  where $\Gamma$ surrounds $z_0$ inside the domain of analyticity of
  $f$. If $\Gamma = \Gamma_\delta := \{ z \in \C : \abs{z - z_0} =
  \delta \}$, we get the estimate
  \begin{equation*}
    \norm{f'(z_0)} \leq  \frac{1}{2\pi} 
    \sup_{\xi \in \Gamma_\delta}{\norm{f(\xi)}}
    \cdot \frac{1}{\delta^2}\cdot 2 \pi \delta 
    = \frac{1}{\delta}\sup_{\xi \in \Gamma_\delta}{\norm{f(\xi)}}
  \end{equation*}
  Note that $ (n + 1) z Q (1 + z Q)^n = z \tfrac{d}{dz} (1 + z Q)^{n +
    1}$ for all $z \in \C$. Applying the above estimate gives for all
  $z \in \mathcal B_Q^\circ$ and $n \in \N$
  \begin{equation*}
    (n + 1) \norm{z Q (1 + z Q)^n}
    \leq \abs{z} \cdot \sup_{\xi \in \Gamma_\delta}{\norm{(1 + \xi Q)^{n + 1}}}
\leq \abs{z} \cdot \sup_{\xi \in \Gamma_\delta}{M_{\xi}}
  \end{equation*}
  where $\delta = \mathop{dist}(z, \partial \mathcal B_Q^\circ)$.
  Because $\Gamma_\delta$ is a compact subset of $\mathcal B_Q^\circ$,
  the claim follows from Proposition \ref{MaximumPropertyProp}.
\end{proof}
An alternative proof for Proposition \ref{InterioInclusionProp} can be
based on Proposition \ref{KreissSetInteriorProp}.
\begin{theorem} \label{InteriorBQoThm}  
  Let $Q \in \BLO(X)$ be a quasinilpotent operator.  Then {\rm (i)}
  $\mathcal R_Q = \mathcal B_Q^\circ \cup \{ 0 \}$ and $\partial
  \mathcal R_Q = \partial \mathcal B_Q^\circ \cup \{ 0 \}$; {\rm(ii)}
  $\mathcal R_Q \neq \{0 \}$ if and only if the set $\mathcal
  B_Q^\circ$ is a non-empty open sector; and {\rm(iii)}
  either $\mathcal R_Q = \{ 0 \}$, $\mathcal R_Q = \C$ (i.e., $Q =
  0$), or there exists $\phi \in [-\pi, \pi)$ and $\theta \in
    (0,\pi/2]$ such that $\mathcal R_Q = e^{i \phi} \Sector{\theta}
  \cup \{ 0 \}$.  
\end{theorem}
\begin{proof}
 We know by claim \eqref{InclusionThmClaim2} of Theorem
 \ref{InclusionThm} and Proposition \ref{InterioInclusionProp} that
  \begin{equation} \mathlabel{InteriorBQoThmEq1} 
    \mathcal B_Q^\circ
    \cup \{ 0 \} \subset \mathcal B_Q \cap \mathcal T_Q = \mathcal R_Q
    \subset \mathcal B_Q.  
  \end{equation}
  It follows that $\mathcal B_Q^\circ = \mathcal R_Q^\circ$ where
  $\mathcal R_Q^\circ$ denotes the open interior of $\mathcal R_Q$.
  Because $\mathcal B_Q$ is convex by Proposition
  \ref{TrivialitiesProp}, so is its interior $\mathcal
  B_Q^\circ$. Because $\mathcal R_Q$ consists of full rays by claim
  \eqref{InclusionThmClaim7} of Theorem \ref{InclusionThm}, so does
  its interior $\mathcal R_Q^\circ$.  We conclude that either
  $\mathcal R_Q^\circ = \emptyset$, or it is an open, convex set that
  consists of full rays.  In other words, $\mathcal R_Q^\circ$ is an
  open sector if $\mathcal R_Q^\circ \neq \emptyset$.  It follows from
  this that
\begin{equation*}
  \mathcal R_Q = \mathcal R_Q^{\circ} \cup \{ 0 \} \cup E \quad \text{
    with } \quad E = \cup_{\phi \in A} e^{i \phi} \R_+
\end{equation*}
where $A \subset [-\pi,\pi)$, $\mathcal R_Q^{\circ} \cap E =
  \emptyset$, and $E^\circ = \emptyset$. Thus $e^{i \phi} \R_+
  \subset \partial \mathcal R_Q$ for all $\phi \in A$.

  We proceed to show that $\mathcal R_Q = \mathcal R_Q^\circ \cup \{ 0
  \}$ by excluding the set $E$.  It is clear that this equality may
  fail only if $e^{i \phi} \R_+ \subset \partial \mathcal R_Q \cap
  \mathcal R_Q$ for some $\phi \in [-\pi,\pi)$. This is, however,
    impossible by the equivalence of
    \eqref{NevanlinnaLubichPropCondition2} and
    \eqref{NevanlinnaLubichPropCondition3a} of Proposition
    \ref{NevanlinnaLubichProp}, implying that any ray in $\mathcal
    R_Q$ is in fact contained in an open sector in $\mathcal R_Q$.
    This proves that $\mathcal R_Q = \mathcal B_Q^\circ \cup \{ 0 \}$
    and hence $\partial \mathcal R_Q = \partial \mathcal B_Q^\circ$ if
    $\mathcal R_Q \neq \{ 0 \}$.  This proves {\rm(i)}, and the latter
    two claims are consequences of this.
\end{proof}
Thus, the set $\mathcal R_Q$ is convex sector whose central angle
plays such an important role that it deserves a name:
\begin{definition}
  \label{RittAngleDef}
  Let $Q \in \BLO(X)$ be a quasinilpotent operator. If $\mathcal R_Q
  \neq \{ 0 \}$, we call the angle $\theta \in (0,\pi/2]$ the
  \emph{Ritt angle} of operator $Q$ if $\mathcal R_Q = e^{i \phi}
  \Sector{\theta} \cup \{ 0 \}$ for some $\phi \in [-\pi, \pi)$.  If
    $\mathcal G_Q = e^{i \phi} \R_+$for some $\phi \in [-\pi, \pi)$,
      then we say that the Ritt angle of $Q$ equals $0$.  (If
      $\mathcal G_Q = \{ 0 \}$, the Ritt angle of $Q$ is not
      defined. )
\end{definition}

It is now possible to improve claim \eqref{InclusionThmClaim1} of Theorem 
\ref{InclusionThm} a bit:
\begin{proposition}
\label{BoundaryExcludedProp}
  We have $\overline{\mathcal B_Q} \subset \mathcal A_Q$ and hence
  $\partial \mathcal B_Q \cap \mathcal T_Q  = \{ 0 \}$ for any
  quasinilpotent operator $Q \in \BLO(X)$.
\end{proposition}
\begin{proof}
  We show first that $\overline{\mathcal B_Q} \subset \mathcal A_Q$.
  This claim clearly holds when $\mathcal B_Q^\circ = \C$ but this
  situation happens only if $Q = 0$ by Remark 
  \ref{ZeroBoundaryPointProp}. 

  Let us first consider the case $\mathcal B_Q^\circ = \emptyset$.
  Then by convexity, $\mathcal B_Q \subset e^{i \phi} \clos{\R}_+$ for
  some $\phi \in [-\pi, \pi)$. If $\mathcal B_Q = \{ 0 \}$, there is
    nothing to prove. Otherwise $\mathcal B_Q$ is a possibly
    non-closed, possibly finite interval whose one end is at the
    origin, and then $\overline{\mathcal B_Q} \subset \mathcal B_Q
    \R_+ \subset \mathcal A_Q^k$ for all $k \in \N$ by claim
    \eqref{InclusionThmClaim1} of Theorem \ref{InclusionThm}.
  
  Suppose now that $\mathcal B_Q^\circ \neq \emptyset$.  By Theorem
  \ref{InteriorBQoThm}, we have $\partial \mathcal B_Q = \partial
  \mathcal R_Q$ (that is a set consisting of two rays), and thus claim
  \eqref{InclusionThmClaim4} of Theorem \ref{InclusionThm} implies
  that $\partial \mathcal B_Q \subset \mathcal A_Q$ by proximity.  We
  conclude that $\overline{\mathcal B_Q} = \mathcal B_Q \cup \partial
  \mathcal B_Q \subset \mathcal A_Q$ as claimed.
  
  We have now shown that $\partial \mathcal B_Q \subset \mathcal A_Q$
  for all quasinilpotent operators $Q$. If $z$ is in $\partial
  \mathcal B_Q \cap \mathcal T_Q$, then $z \in \mathcal A_Q \cap
  \mathcal T_Q = \mathcal R_Q$ by claim \eqref{InclusionThmClaim2b} of
  Theorem \ref{InclusionThm}. We also have $z \in \partial \mathcal
  R_Q$ because $\partial \mathcal R_Q = \partial \mathcal B_Q$, and
  hence $z = 0$ follows since $\mathcal R_Q = \mathcal B_Q^\circ \cup
  \{ 0 \}$ by Theorem \ref{InteriorBQoThm}.
\end{proof}
\begin{proposition} \label{UniformlyBPonRayProp}
  If $\sup_{r \geq 0}{N_{rz}} < \infty$ in \eqref{MaxPowerFunction}
  for some $z \in \mathcal B_Q$, then $z \in \mathcal R_Q$. In
  particular, $\sup_{r \geq 0}{N_{rz}} = \infty$ for $z \in \partial
  \mathcal B_Q$.
\end{proposition}
\begin{proof}
  Defining $M := \sup_{r \geq 0}{N_{rz}}$, we see that $T(rz)$ satisfies
the Kreiss condition
\begin{equation*}
\norm{\left (\lambda - T(r z) \right )^{-1}} \leq \frac{M}{\abs{\lambda} - 1}
\text{ for all } \abs{\lambda} > 1 \text{ and } r > 0.
\end{equation*}
After some manipulations, this is equivalent with
\begin{equation}
\mathlabel{UniformlyBPonRayPropEq1}
\norm{\left (I - s z Q \right )^{-1}} \leq \frac{M r}{\abs{r + s} - \abs{s}}
\text{ for all } r > 0 \text{ and } s \text{ satisfying } \abs{r + s} >  \abs{s}.
\end{equation}
Now, $\abs{r + s} > \abs{s}$ for $r > 0$ if and only if $\Re{s} >
-r/2$.  Letting $r \to +\infty$ in \eqref{UniformlyBPonRayPropEq1}
implies $\norm{\left (I - s z Q \right )^{-1}} \leq M$ for all $s \in
\C_+$, and it follows from Proposition \ref{RittSetProp} that $z \in
\mathcal R_Q$.
\end{proof}

\ConsiderThis{Consider this:
  \begin{enumerate}
  \item Considering Prop \ref{UniformlyBPonRayProp}: What if $\sup_{r
    \in [0,R]}{N_{rz}} < \infty$? What do you get then? Try $r > 0$
    and $s = i \omega$. Then you will get a lot of growth on the
    imaginary axis... hence no Ritt.
  \item Interior of $\mathcal G_Q$ is what you expect it to be.
    Show that there is no open subset of $\mathcal G_Q \setminus \mathcal B_Q$. 
  \item suppose that $f(t) = e^{t Q}$ is a bounded, analytic, operator
    valued function defined on domain $\mathcal G_Q^\circ \cap \{ 0
    \}$. Can it happen that this $f$ has a bounded extension to
    $\overline{\mathcal G_Q}$?
  \end{enumerate}
}

\section{\label{InteriorFSec} Interior points of $\mathcal G_Q$, $\mathcal K_Q$, and $\mathcal K_Q^\infty$}

Let us recall \cite[Definition 3.7.3]{A-B-H-N:VVLTCP} of bounded
analytic semigroups:
\begin{definition} \label{AnalSemiDef}
  We say that a strongly continuous semigroup $S(t)$, $t \in
  \clos{\R}_+$, is a \emph{bounded analytic (holomorphic)
    semigroup} of angle $\theta \in (0, \pi/2]$ if $S(t)$ has a
  bounded, analytic extension to $\Sector{\theta'}$ for all $\theta' \in
  (0,\theta)$.
\end{definition}
In the context of this paper, the semigroup is given by $S_z(t) = e^{t
  z Q}$ for quasinilpotent $Q$ and $z \in \mathcal G_Q$.  Let us
recall a classical result from function theory that we need in proving
Theorem \ref{SemigroupInteriorThm}:
\begin{proposition} \label{PhragmenLindelofThm} 
Let $X$ be a Banach space.  Suppose that $A, \omega, \tau > 0$, $
\theta \in (0,\pi)$, and that $\Sector{\theta}$ is defined by
\eqref{BalancedOpenSector}. Suppose that th $\mathcal L(X)$-valued
function $f$ is analytic in $\Sector{\theta}$ and continuous in
$\overline{\Sector{\theta}}$.  Suppose that the inequalities
\begin{equation} \mathlabel{PhragmenLindelofThmEq1} 
  2 \theta < \frac{\pi}{\omega} \quad \text{ and } \quad 
  \norm{f(z)} \leq A e^{\tau \abs{z}^\omega} \text{ for all } z \in \Sector{\theta} 
\end{equation}
hold. If $\norm{f(z)} \leq 1$ for all $z \in \partial \Sector{\theta} $, then
also $\norm{f(z)} \leq 1$ for all $z \in \Sector{\theta}$.
\end{proposition}
\begin{proof}
  By considering the functions $z \mapsto \left <x^*, f(z) x \right
  >_{X^*,X}$ for $x \in X$ and $x^* \in X^*$, the claim can be reduced
  to the scalar case which we prove next.  Let $\theta < \frac{\pi}{2
    \omega}$ and take $z \in \Sector{\theta}$.  Then for any $\beta$
  satisfying $\omega < \beta < \frac{\pi}{2 \theta}$ we have the
  inequality
  \begin{equation*}
   \Re{(z + 1)^\beta} > \abs{z + 1}^\beta \delta 
  \end{equation*}
  where $\delta := \cos{\beta \theta} > 0$ since $\abs{\arg{(z + 1)}}
  < \abs{\arg{z}} \leq \theta$ and $0 < \beta \theta < \pi/2$. Define
  now $h_\epsilon(z) = e^{- \epsilon (z + 1)^\beta}$ for any $\epsilon
  > 0$.  This is an analytic function for $z \in \Sector{\theta} \neq \C$,
  and it is continuous in $\overline{\Sector{\theta}}$.  Then for all $z \in
  \overline{\Sector{\theta}}$ we have the estimate
  \begin{equation*}
    \abs{f(z) h_\epsilon(z)} 
    \leq Ae^{\tau \abs{z}^\omega - \epsilon \Re{(z + 1)^\beta}}
    \leq Ae^{\tau \abs{z}^\omega - \epsilon  \delta \abs{z + 1}^\beta}
\to 0 \text{ as } \abs{z} \to \infty.
  \end{equation*}
The maximum modulus theorem says that
\begin{equation*}
  \max_{z \in \overline{\Sector{\theta}}} \abs{f(z) h_\epsilon(z)}
=   \max_{z \in \partial {\Sector{\theta}}} \abs{f(z) h_\epsilon(z)}
\leq \max_{z \in \partial {\Sector{\theta}}} \abs{h_\epsilon(z)}
\end{equation*}
for all $\epsilon > 0$. Letting $\epsilon \to 0+$ completes the proof.
\end{proof}
  Note that $f(z) = e^z$ satisfies $\abs{f(z)} \leq A e^{\tau
    \abs{z}^\omega}$ for $A = \omega = \tau = 1$ and all $z \in C$. If
  $2 \theta = \pi/\omega = \pi$ and $\Sector{\theta} = \C_+$, we see
  that $f$ is bounded on $\partial \Sector{\theta} = i \R$ but not on
  $\Sector{\theta}$.  Hence the inequality $2 \theta <
  \frac{\pi}{\omega}$ cannot be replaced by the weaker $2\theta \leq
  \frac{\pi}{\omega}$ in Proposition \ref{PhragmenLindelofThm}.

\begin{theorem} \label{SemigroupInteriorThm}
  Let $Q \in \BLO(X)$ be a quasinilpotent operator. Then the following
  are equivalent for $z \in \C$:
  \begin{enumerate}
  \item \label{SemigroupInteriorThmClaim1} $z \in  \mathcal G_Q^\circ$;
  \item \label{SemigroupInteriorThmClaim2} $z \neq 0$ and the
    operator $zQ$ generates an analytic semigroup $t \mapsto e^{t z
      Q}$; and
  \item \label{SemigroupInteriorThmClaim3} $z \in  \mathcal R_Q\setminus\{0\}$.
  \end{enumerate}
\end{theorem}
\noindent Thus $\mathcal R_Q = \mathcal G_Q^\circ \cup \{ 0 \}$. Moreover, we
have $\mathcal R_Q = \{ 0 \}$ if and only if $\mathcal G_Q \subset
e^{i \phi} \clos{\R}_+$ for some $\phi \in [-\pi, \pi)$.
\begin{proof}
  \eqref{SemigroupInteriorThmClaim1} $\Rightarrow$
  \eqref{SemigroupInteriorThmClaim2}: Take $z \in G_Q^\circ$, and let
  $0 < \theta < \pi/2$ such that both $ze^{i\theta}, ze^{-i\theta} \in
  G_Q^\circ$. It follows that the analytic function $g(t) := e^{t zQ}$
  is bounded on both rays $t \in e^{\pm i \theta} \R_+$.  Since
  $\norm{e^{t zQ}} \leq e^{ \tau \abs{t}^\omega}$ with $\tau =
  \norm{zQ}$ and $\omega = 1$, Proposition \ref{PhragmenLindelofThm}
  implies now that $\sup_{t \in \Sector{\theta}}{\norm{g(t)}} < \infty$;
  hence $zQ$ generates a bounded analytic semigroup with angle $\geq
  \theta$.
  
  \eqref{SemigroupInteriorThmClaim2} $\Leftrightarrow$
  \eqref{SemigroupInteriorThmClaim3}: This follows from Proposition
  \ref{RittSetProp} because the semigroup $S_z(t) := e^{t z Q}$
  satisfies the conditions of Definition \ref{AnalSemiDef} for some
  angle in $(0, \pi/2]$ if and only if $\sup_{\Re{s}> 0}{\norm{(1 - s
      zQ)^{-1}}} < \infty$; see, e.g., \cite[Corollary
  3.7.12]{A-B-H-N:VVLTCP}.
  
  \eqref{SemigroupInteriorThmClaim3} $\Rightarrow$
  \eqref{SemigroupInteriorThmClaim1}: Because $\mathcal B_Q \subset
  \mathcal G_Q$ by claim \eqref{InclusionThmClaim1} of Theorem
  \ref{InclusionThm}, also their interiors satisfy $\mathcal B_Q^\circ
  \subset \mathcal G_Q^\circ$.  Thus by Theorem
  \ref{InteriorBQoThm} we get $\mathcal R_Q \setminus \{0\} =
  \mathcal B_Q^\circ \subset \mathcal G_Q^\circ$.
\end{proof}

\begin{proposition} \label{KreissSetInteriorProp}
A quasinilpotent $Q \in \BLO(X)$ satisfies $\mathcal R_Q = \mathcal
K_Q^{\circ} \cup \{ 0 \}$.
\end{proposition}
\begin{proof}
Let $s = r e^{i \phi} \in \C_+$ and define $\theta$ by $s + 1 = \abs{s
  + 1}e^{i \theta}$. Then $\abs{s + 1} > \abs{s}$, $\Re{s} > 0$, and
we obtain
\begin{equation*}
\frac{1}{\abs{s + 1} - \abs{s}} < \frac{2 \abs{s + 1}}{2 \Re{s} + 1}
< \frac{2 \abs{s + 1}}{\Re{(s + 1)}} = \frac{2}{\cos{\theta}} \leq
\frac{2}{\cos{\phi}} 
\end{equation*}
where the last inequality holds because $0 \leq \abs{\theta} \leq
\abs{\phi} < \pi / 2$ and thus $0 < \cos{\phi} \leq \cos{\theta}$.
For $z \in \mathcal K_Q$ we thus have a constant $M_z < \infty$ such
that
\begin{equation*}
  \norm{(I - szQ)^{-1}} \leq \frac{M_z}{\cos{\phi}} \quad \text{ for all } \quad  s =  r e^{i \phi} \in \C_+.
\end{equation*}
If $z \in \mathcal K_Q^\circ$, we have a $\delta > 0$ such that
$e^{\pm i \delta} z \in \mathcal K_Q$.  Writing $s_{\pm} := s e^{\pm
  i \delta} = \abs{s} e^{i (\phi \pm \delta)}$ we see that there is
a constant $M < \infty$ such that 
\begin{equation*}
  \norm{(I - s_{\pm} zQ)^{-1}} \leq \frac{M}{\cos{\phi}} \quad \text{ for all } \quad  s =  r e^{i \phi} \in \C_+.
\end{equation*}
Varying $\phi$ through the interval $(- \pi / 2 + \delta/2, \pi / 2 -
\delta/2)$ (in which $1/\cos{\phi}$ is bounded) proves that $\sup_{s'
  \in \C_+} {\norm{(I - s' zQ)^{-1}}} < \infty$, and by Proposition
\ref{RittSetProp} we have $z \in \mathcal R_Q$. The claim follows
since $\mathcal R_Q \setminus \{ 0 \}$ is open by Theorem
\ref{InteriorBQoThm}.
\end{proof}

\begin{proposition} \label{IteratedKreissSetConvexProp}
The set $\mathcal K_Q^\infty$ is convex for all quasinilpotent $Q \in \BLO(X)$.
\end{proposition}
\begin{proof}
We have $\mathcal R_Q \subset \mathcal K_Q^\infty \subset \mathcal
G_Q$ by Theorem \ref{InclusionThm}.  If $\mathcal R_Q = \{0 \}$, then
$\mathcal G_Q \subset e^{i \phi} \clos{\R}_+$ for some $\phi \in
[-\pi,\pi)$ by Theorem \ref{SemigroupInteriorThm}. Then there is
  nothing to prove because $\mathcal K_Q^\infty$ is star-like (hence,
  a convex subset of $e^{i \phi} \clos{\R}_+$) by Proposition
  \ref{TrivialitiesProp}.

If $\mathcal R_Q \neq \{0 \}$, the sets $\mathcal R_Q$ and $\mathcal
G_Q$ are sectors with the same interior $\mathcal R_Q \setminus \{ 0
\}$.  Thus, $\mathcal K_Q^\infty$ may fail to be convex only if the
set $\mathcal K_Q^\infty \cap \partial \mathcal G_Q$ is not star-like
which, however, is excluded by Proposition \ref{TrivialitiesProp}.
\end{proof}

\begin{theorem}
  \label{OpenInteriorCLosureThm}
  For any quasinilpotent operator $Q \in \BLO(X)$ we have 
  \begin{equation*}
    \mathcal R_Q = \mathcal B_Q^\circ \cup \{0 \}  =  \mathcal G_Q^\circ \cup \{0 \} =
    \mathcal K_Q^\circ \cup \{0 \} = (\mathcal
    K_Q^\infty)^\circ \cup \{0 \}
  \end{equation*}
If $\mathcal R_Q \neq \{ 0 \}$, then $\overline{\mathcal B_Q} =
\overline{\mathcal G_Q} = \overline{\mathcal K_Q^\infty} =
\overline{\mathcal R_Q}$ and $\partial \mathcal B_Q = \partial
\mathcal G_Q = \partial \mathcal K_Q^\infty = \partial \mathcal R_Q$.
\end{theorem}
\begin{proof}
  The first equality is just a composition of claim
  \eqref{InclusionThmClaim8b} of Theorem \ref{InclusionThm} (giving
  first $\mathcal B_Q^\circ \subset (\mathcal K_Q^\infty)^\circ
  \subset \mathcal G_Q^\circ$ and hence the equality of interiors),
  Theorems \ref{InteriorBQoThm}, \ref{SemigroupInteriorThm}, and
  Proposition \ref{KreissSetInteriorProp}.

  By what we have already proved, all of the sets $\mathcal B_Q$,
  $\mathcal G_Q$, $\mathcal K_Q^\infty$, and $\mathcal R_Q$ are convex
  with the same interior $\mathcal R_Q \setminus \{ 0 \}$.  To
  complete the proof, we show that for any convex set $K \subset \C$
  with a \emph{nonempty} interior $K^\circ$, the closures of $K^\circ$
  and $K$ coincide.
  
  Suppose that the inclusion $\overline{K^\circ} \subset \overline{K}$
  of closures is strict, and choose $z \in \overline{K} \setminus
  \overline{K^\circ}$.  Denote $\delta := \mathop{dist}\left (z,
    K^\circ \right ) = \mathop{dist}\left (z, \overline{K^\circ}
  \right ) > 0$. Because $z \in \overline{K}$, there exists a $z' \in
  K$ such that $\abs{z - z'} < \delta / 3$, and thus
  $\mathop{dist}\left (z', K^\circ \right ) > \delta/2$.  Because $K$
  is convex, it follows that
\begin{equation*}
  \mathop{conv}\left ( z',  D \right )
  := \{ \alpha z' + (1 - \alpha) z : \alpha \in (0,1) \text{ and } 
  z \in D  \} \subset K.
\end{equation*}
where $D$ is any open disc of positive radius contained in $K^\circ$
(here we use the assumption that $K^\circ$ is nonempty).  Since the
set $\mathop{conv}\left ( z', D \right )$ is an open cone, it contains
only interior points of $K$, and some of them are arbitrarily close to
its vertex $z'$.  This is a contradiction against $\delta > 0$.
\end{proof}
There is an observation concerning the endpoints of $\mathcal B_Q \cap
\partial \mathcal B_Q$:
\begin{proposition} \label{BoundingRayProp}
  We have $\mathcal T_Q^{1/2} \cap \mathcal G_Q \cap \partial \mathcal
  G_Q \subset \mathcal B_Q \cap \partial \mathcal B_Q \subset
  \overline{\mathcal T_Q^{1/2} \cap \mathcal G_Q \cap \partial
    \mathcal G_Q}$. One of the inclusions is strict if the set
  $\mathcal B_Q \cap \partial \mathcal B_Q$ is non-trivial (i.e.,
  neither $\{ 0 \}$, a full ray of infinite length, or a pair of such
  full rays).
\end{proposition}
\begin{proof}
   If $\mathcal R_Q \neq \{ 0 \}$ the first inclusion follows because
   $\mathcal T_Q^{1/2} \cap \mathcal G_Q \subset \mathcal B_Q$ by
   claim \eqref{InclusionThmClaim3} of Theorem \ref{InclusionThm} and
   $\partial \mathcal G_Q = \partial \mathcal B_Q$ by Theorem 
   \ref{OpenInteriorCLosureThm}.  If $\mathcal R_Q = \{ 0 \}$, use
   $\mathcal G_Q \cap \partial \mathcal G_Q = \mathcal G_Q$ and
   $\mathcal B_Q \cap \partial \mathcal B_Q = \mathcal B_Q$ instead.

  For the second inclusion, we argue as follows: If $z \in \mathcal
  B_Q \cap \partial \mathcal B_Q$, we have for all $\alpha \in [0,1)$
    the inclusion $\alpha x \in \mathcal T_Q^{1/2} \cap \mathcal G_Q
    \cap \partial \mathcal G_Q$ because $\alpha x \in \mathcal B_Q
    \subset \mathcal G_Q$ and $\alpha x \in \mathcal T_Q^{1/2}$ by
    claims \eqref{InclusionThmClaim1} and \eqref{InclusionThmClaim5}
    of Theorem \ref{InclusionThm}, and $\partial \mathcal B_Q =
    \partial \mathcal G_Q$ by Theorem \ref{OpenInteriorCLosureThm}
    if $\mathcal R_Q \neq \{ 0 \}$.  If $\mathcal R_Q = \{ 0 \}$,
    proceed as above. Letting $\alpha \to 1-$ proves that $x \in
    \overline{\mathcal T_Q^{1/2} \cap \mathcal G_Q \cap \partial
      \mathcal G_Q}$ in both cases.

For contradiction, assume that at least one side of $\mathcal B_Q \cap
\partial \mathcal B_Q $ is a closed ray of finite positive length, say
$[0, z]$ for $z \in \C$.  If $z \in \mathcal T_Q^{1/2}$, it follow
from \cite[Theorem 1.2]{ND:TRRCPBO} that $z/\beta \in \mathcal B_Q
\cap \partial \mathcal B_Q$ for some $\beta \in (0,1)$ which is
impossible.
\end{proof}
Note that by \cite[Theorem 1.2]{ND:TRRCPBO}, claim
\eqref{InclusionThmClaim5} of Theorem \ref{InclusionThm} cannot be
improved to the inclusion $\mathcal B_Q = \mathcal B_Q[0,1] \subset
\mathcal T_Q^{1/2}$ whenever $\mathcal B_Q$ does not consist of full
rays.

\ConsiderThis{
\begin{enumerate}
\item write a note about an operator for which $\mathcal G_Q \supset
  \R$. If $Q \neq 0$, then either $\mathcal G_Q = \R, \{ \Im z \geq 0
  \}, \{ \Im z \leq 0 \}$. In all of these cases, it is a closed set.
\end{enumerate}
}

\section{\label{TauberianGrowthSec} Growth on rays throught Tauberian sets}

We now give estimates on the growth of resolvents and semigoups on the
rays that intersect the Tauberian sets $\mathcal T_Q$ or $\mathcal
T_Q^{1/2}$.  
\begin{proposition}
  \label{TauberianGrowthbounds}
  Let $Q \in \BLO(X)$ be a quasinilpotent operator and $z \in \C$.
  \begin{enumerate}
  \item \label{TauberianGrowthboundsClaim1} If $z \in \mathcal
    T_Q^{1/2}$, then there are constants $C_r, C_g < \infty$ (both
    depending on $z$ and given by
    \eqref{TauberianGrowthboundsClaim1Constant}) such that
    \begin{align} 
      \mathlabel{TauberianGrowthboundsClaim1EqA}
      \norm{(1 - s z Q)^{-1}} 
      & \leq C_r (s + 1)^{1/2} + 1 \quad
      \text{ for all } \quad s \geq 0;  \quad \text{ and } \\
      \mathlabel{TauberianGrowthboundsClaim1EqB} 
      \norm{e^{t z Q}} &
      \leq C_g (t^{1/2} + 1) \quad \text{ for all } \quad t \geq 1.
    \end{align}
  \item   \label{TauberianGrowthboundsClaim2}
    If $z \in \mathcal T_Q$, then there are constants $D_r, D_g < \infty$
    (both depending on $z$  and given by
    \eqref{TauberianGrowthboundsClaim2Constant}) such that
    \begin{align} \mathlabel{TauberianGrowthboundsClaim2EqA}
      \norm{(1 - s z Q)^{-1}} & \leq D_r \ln{(s + 1)} + 1 \quad \text{
        for all } \quad s > 0; \quad
      \text{ and } \\
      \mathlabel{TauberianGrowthboundsClaim2EqB} \norm{e^{t zQ}} &
      \leq D_g (\ln{(t + 1/e)} + 4)  \quad \text{ for all } \quad t > 1.
    \end{align}
  \end{enumerate}
\end{proposition}
\begin{proof}
  The proof is based on estimating the growth of powers of $T(z) = 1 +
  zQ$, $z \in \C$, using the identities
\begin{align*}
  \left (1 - \frac{\xi z}{1 - \xi }Q \right )^{-1} & =
  (1 - \xi ) \left (I - \xi T(z) \right )^{-1} \\
  & = I - \sum_{j \geq 1}{\xi^{j} (I - T(z)) T(z)^{j-1}} \quad
  \text{ for } \quad \xi \in \D, \\
  T(z)^k &= I - \sum_{j = 0}^{k - 1}{ \left (I - T(z) \right ) T(z)^j},
  \quad \text{ and }  
\end{align*}
\begin{align*}
  e^{t zQ} & = e^{-t} e^{t T(z)} = e^{-t} \sum_{j \geq
    0}{\frac{t^j T(z)^j}{j!}} \quad \text{ for } \quad t\in \C.
\end{align*}
Claim \eqref{TauberianGrowthboundsClaim1}: We prove the claims for $z
\in \mathcal T_Q^{1/2}$ with the constants
\begin{equation} \mathlabel{TauberianGrowthboundsClaim1Constant}
\begin{aligned}
  C_r & := \pi^{1/2} M \quad \text{ and } \quad C_g := (2 M + 1)e^{1/2} \quad \text{ where
  } \\
M & := \sup_{j \geq 0}{(j + 1)^{1/2} \norm{(I - T(z)) T(z)^{j}}} < \infty.
\end{aligned}  
\end{equation}
For all $\xi \in [0,1)$ we have
\begin{equation} \mathlabel{TauberianGrowthboundsClaim1Eq1}
  \norm{\left (1 - \frac{\xi z}{1 - \xi }Q  \right )^{-1}}
\leq 1 +  \frac{C_r}{\pi^{1/2}}\sum_{j \geq 1}{\frac{\xi^{j}}{j^{1/2}}}
\end{equation}
where we majorize the sum by an integral, noting that $\xi \in (0,1)$
and $-\infty < \ln{\xi} < 0$:
\begin{align*}
  \sum_{j \geq 1}{\frac{\xi^{j}}{j^{1/2}}} & \leq 
  \int_{0}^\infty{\frac{\xi^{j} \, dj}{j^{1/2}}} =
  \int_{0}^\infty{\frac{\xi^{x^2} \, 2x dx}{(x^2)^{1/2}}}
 =  \frac{1}{\abs{\ln{\xi}}^{1/2}} 
\int_{-\infty}^{\infty}{e^{- y^2} \,  dy} = \sqrt{\frac{\pi}{\abs{\ln{\xi}}}}.
\end{align*}
Now for $s \in \R_+$ we have $s = \frac{\xi}{1 - \xi}$ if and only if
$\xi = \frac{s}{s + 1}$. Thus $\abs{\ln{\xi}} = -\ln{\xi} =
\ln{\frac{s + 1}{s}} = \ln{(s + 1)}- \ln{s} \geq \frac{1}{s + 1}$ by
the mean value theorem, and thus \eqref{TauberianGrowthboundsClaim1EqA}
follows. 

To estimate $\norm{e^{tzQ}}$, we note  that $\norm{T(z)^j} \leq 1
+ 2 M(j - 1)^{1/2} \leq (2M + 1) j^{1/2}$ for $j \geq 1$, and we get
\begin{equation} \mathlabel{TauberianGrowthboundsClaim1Eq2}
\norm{e^{tzQ}} \leq (2M + 1) e^{-t}
\left( 1 + \sum_{j \geq 1}{\frac{t^{j} \sqrt{j}}{j!}} \right )
\text{ for } t \geq 0.  
\end{equation}
Assume that $t > 1/e$, and let $J$ be the integer satisfying $1
\leq J \leq et < J + 1$. Then
\begin{align*}
  & \sum_{j \geq 1}{\frac{t^{j} \sqrt{j}}{j!}}  \leq \sum_{j =
    1}^J{\frac{t^{j} \sqrt{j}}{j!}} + \sum_{j \geq J + 1}{\frac{t^{j}
      \sqrt{j}}{j!}}
  \leq J^{1/2} e^{t}  + \sum_{j \geq J + 1}{\frac{t^{j} }{(j-1)!}} \\
  & \leq J^{1/2} e^{t}
  + \frac{t^{J + 1}}{J!} \left (1 + \frac{t}{J + 1} + \frac{t^2}{(J + 1)(J + 2) } \cdots  \right ) \\
  & \leq J^{1/2} e^{t} + \frac{t^{J + 1}}{J!}  \sum_{k \geq 0}{\left (
      \frac{t}{J + 1} \right )^k} =  J^{1/2} e^{t} + \frac{t^{J +
      1}}{J!} \left ( 1 - \frac{t}{J + 1} \right )^{-1}.
\end{align*}
Since $J \geq 1$ we get
\begin{equation*}
  \frac{t^{J + 1}}{J!} \leq \frac{\left (\frac{J + 1}{e} \right )^{J + 1}}{
    \sqrt{2 \pi J} \left
      (\frac{J}{e} \right )^J} \leq 
  \frac{1}{\sqrt{2 \pi J}} \left(J + 1 \right) \cdot 
  e^{-1}\left(1 + \frac{1}{J} \right)^J 
  \leq \frac{1}{\sqrt{2 \pi}} \frac{J + 1}{J^{1/2}}
\end{equation*}
where we used Stirling's approximation $J! > \sqrt{2 \pi J} \left
  (\frac{J}{e} \right )^J$. Noting that $et - 1 < J \leq e t$, $\tfrac{t}{J +
  1} < e^{-1}$ and hence  $\frac{1}{\sqrt{2 \pi}} \left ( 1 - \frac{t}{J + 1}
\right )^{-1} < 1$,  we get
\begin{equation*}
  \sum_{j \geq 1}{\frac{t^{j} \sqrt{j}}{j!}}  \leq
  (e t )^{1/2} e^t + (e t + 1)(e t - 1)^{-1/2}. 
\end{equation*}
Now \eqref{TauberianGrowthboundsClaim1EqB} follows from 
\eqref{TauberianGrowthboundsClaim1Eq2} and $\max_{t \geq
  1}{e^{-t}\left ( 1 + 2(e t + 1)(e t - 1)^{-1/2}\right )} <
\sqrt{2}$.

Claim \eqref{TauberianGrowthboundsClaim2}: Take $z \in \mathcal T_Q$
and define the constants
\begin{equation} \mathlabel{TauberianGrowthboundsClaim2Constant}
  D_r := \sup_{j \geq 0}{(j + 1) \norm{(I - T(z)) T(z)^{j}}} < \infty \quad \text{and} \quad 
  D_g := D_r + 1.
\end{equation}
The resolvent estimate of claim \eqref{TauberianGrowthboundsClaim2} is
proved as above but instead of equation
\eqref{TauberianGrowthboundsClaim1Eq1} we now compute for $\xi \in
      [0,1)$ the sum
\begin{equation*}
 \sum_{j \geq 1}{\frac{\xi^{j}}{j}} = \int_{0}^\xi{\sum_{j \geq 1}{x^{j-1}} \, d x}
= \int_{0}^\xi{\frac{d x }{1 - x} } = - \ln{(1 - \xi)} = \ln{(s + 1)}.  
\end{equation*}

 It remains to estimate $\norm{e^{tzQ}}$ when $z \in \mathcal T_Q$.
 Again, we have $\norm{T(z)^j} \leq 1 + D_r(1 + \ln{j} ) \leq D_g
 \ln{ej}$ for $j \geq 1$, which gives the estimate just like
 \eqref{TauberianGrowthboundsClaim1Eq2}
\begin{equation} \mathlabel{TauberianGrowthboundsClaim2Eq1}
\norm{e^{tzQ}} \leq D_g e^{-t}
\left( 1 + \sum_{j \geq 1}{\frac{t^{j} \ln{(ej)}}{j!}} \right ) \quad \text{ for all } \quad t \geq 0.  
\end{equation}
Assume again that $t > 1/e$, and let $J$ be the integer satisfying $1
\leq J \leq e t < J + 1$. Then
\begin{align*}
   \sum_{j \geq 1}{\frac{t^{j} \ln{(ej)} }{j!}}  & \leq \sum_{j = 1}^J
  {\frac{t^{j} \ln{(ej)}} {j!}}
  + \sum_{j \geq J + 1}{\frac{t^{j} \ln{(ej)}}{j!}} \\
  & \leq e^{t} \ln{(e J)} + \frac{\ln{e(J + 1)}}{J + 1} \sum_{j \geq J + 1}{\frac{t^{j} }{(j-1)!}}.
\end{align*}
The sum in the latter term can be estimated by Stirling's
approximation together with the estimates $et < J + 1 \leq e t + 1$,
$\tfrac{t}{J + 1} < e^{-1}$, and $\frac{1}{\sqrt{2\pi}} \left ( 1 -
\frac{t}{J + 1} \right )^{-1} < 1$: 
\begin{align*}
 & \sum_{j \geq J + 1}{\frac{t^{j} }{(j-1)!}} \leq \frac{t^{J +
      1}}{J!} \sum_{j \geq 0}{\left ( \frac{t }{J+1} \right
    )^{j}} < J + 1.
\end{align*}
We get $\sum_{j \geq 1}{\frac{t^{j} \ln{(ej)} }{j!}} \leq e^{t} \ln{(e
  J)} + \ln{e(J + 1)} < e^{t} \ln{e(et + 1)} $, and
\eqref{TauberianGrowthboundsClaim2EqB} follows since $\max_{t \geq
  1}{e^{-t}\ln{e(et + 1)}} < 1$.
\end{proof}

\section{\label{LindelofSec} Consequences of the Phragmen--Lindel\"of theorem}

 Much of the next results are consequences of Phragmen--Lindel\"of
 theorem (see Proposition \ref{PhragmenLindelofThm}) applied to the
 $\BLO(X)$-valued functions $r(s) := (1 - sQ)^{-1}$ (the Fredholm
 resolvent) and $g(t) := e^{t Q}$.  Because $\sigma(Q) = \{ 0 \}$, the
 Gelfand formula for the spectral radius implies that for all $r > 0$,
 there exists $C_r < \infty$ such that $\norm{e^{t Q}} \leq C_r e^{r
   \abs{t}}$ holds for all $t \in \C$; thus the entire function
 $g$ is always of exponential type. Unfortunately, the function
 $r$ does not have the same property without an additional
 compactness assumption on the quasinilpotent $Q$.

\subsection{Conditions for $\mathcal T_Q \subset \mathcal B_Q$}

The purpose of this section is to find sufficient conditions that
imply $\mathcal T_Q \subset \mathcal B_Q$ for a quasinilpotent $Q$;
i.e., that the Tauberian condition \eqref{TauberianConditionT} implies
the power-boundedness of $T(z)$. We already know that
\begin{equation*}
\begin{aligned}
  \mathcal T_Q \subset \mathcal R_Q & \quad \Leftrightarrow \quad
  \mathcal T_Q \subset \mathcal B_Q \quad \Leftrightarrow \quad
  \\ \mathcal A_Q \cap \mathcal T_Q \subset \mathcal B_Q, \quad
  \mathcal T_Q \cap \left (- \mathcal G_Q \right ) & = \emptyset,
  \quad\text{ and } \quad \mathcal T_Q \cap \left ( \mathcal A_Q^c
  \setminus (- \mathcal G_Q ) \right ) = \emptyset.
\end{aligned}  
\end{equation*}
The first equivalence holds by Proposition \ref{BoundaryExcludedProp}
because $\mathcal R_Q = \mathcal B_Q^{\circ} \cup \{ 0 \}$.  Based on
claims \eqref{InclusionThmClaim2} and \eqref{InclusionThmClaim2b} of
Theorem \ref{InclusionThm}, we see that inclusion $\mathcal A_Q \cap
\mathcal T_Q \subset \mathcal B_Q$ always holds. By Proposition
\ref{GelfandHilleProp}, we have $\mathcal T_Q \cap \left (- \mathcal
G_Q \right ) = \emptyset$ for all $Q \neq 0$, too.
 Hence, we conclude that $\mathcal T_Q
\subset \mathcal B_Q$ if and only if
\begin{equation}
  \mathlabel{InclusionEquivalence}
\mathcal T_Q \cap \left ( \mathcal A_Q^c
  \setminus (- \mathcal G_Q ) \right ) = \emptyset,
\end{equation}
and by a similar argument we see $\mathcal T_Q^{1/2} \subset \mathcal
A_Q$ if and only if
\begin{equation}
  \mathlabel{VerifyThisInclusionByPhLind}
  \mathcal T_Q^{1/2} \cap \left ( \mathcal A_Q^c \setminus (- \mathcal G_Q ) \right ) = \emptyset.
\end{equation}
Clearly \eqref{VerifyThisInclusionByPhLind} $\Rightarrow$
\eqref{InclusionEquivalence}, and we look for sufficient conditions to
ensure either of these.  Let us start from a simple observation when
\eqref{VerifyThisInclusionByPhLind} holds trivially:
\begin{proposition}
  Let $Q \in \BLO(X)$ be a quasinilpotent operator whose Ritt angle
  satisfies $\theta > \pi/4$.  Then $\mathcal T_Q^{1/2} \subset
  \mathcal A_Q$ and hence $\mathcal T_Q \subset \mathcal B_Q$.
\end{proposition}
\begin{proof}
  Denote the Ritt angle of $Q$ by $\theta > 0$. Then $\mathcal A_Q$
  contains an open sector $\Sigma := \C_+ \mathcal G_Q$ whose central
  angle is $2 \theta + \pi$; see claim \eqref{InclusionThmClaim4} of
  Theorem \ref{InclusionThm}.  Because the central angle of the sector
  $-\mathcal G_Q$ is $2 \theta$ and $-\mathcal G_Q$ is geometrically
  opposite to $\Sigma$, we conclude that $\mathcal A_Q \cup (-\mathcal
  G_Q)$ contains open sectors whose total central angle is at least
  $\min{(2\pi, 4 \theta + \pi)}$. If now $\theta > \pi/4$, there is no
  room at all left for the set $\mathcal A_Q^c \setminus (- \mathcal
  G_Q )$, and \eqref{VerifyThisInclusionByPhLind} follows.
\end{proof}

\begin{proposition}
   \label{PhragmenLindelofMainProp}
   If $Q \in \BLO(X)$ is a quasinilpotent operator such that the
   origin is an interior point of $\mathop{conv}(\mathcal G_Q \cup
   \mathcal T_Q^{1/2})$, then $Q = 0$.  
If $\mathcal R_Q \neq \{ 0 \}$
   and $Q \neq 0$, then $\mathcal A_Q \not \subset \mathcal
   T_Q^{1/2}$.
\end{proposition}
\noindent Thus, the origin $z = 0$ is a boundary point of all
of the sets $\mathcal B_Q$, $\mathcal G_Q$, $\mathcal K_Q$, $\mathcal
K_Q^\infty$, $\mathcal R_Q$, $\mathcal T_Q^{1/2}$, or $\mathcal T_Q$
if and only if $Q \neq 0$.
\begin{proof}
  If the origin is an interior point of $\mathop{conv}(z_1, z_2, z_3)$
  with $z_1, z_2, z_3 \in \mathcal G_Q \cup \mathcal T_Q^{1/2}$, then
  the three rays $z_1 \R_+$, $z_2 \R_+$, and $z_3 \R_+$ divide $\C$
  into three sectors whose central angles are strictly less than
  $\pi$. On these rays, the entire function $g(t) = e^{tQ}$ grows at
  most like a square root by claim
  \eqref{TauberianGrowthboundsClaim1} of Proposition
  \ref{TauberianGrowthbounds}.  Since the entire function $g$ is of
  exponential type, Proposition \ref{PhragmenLindelofThm} (with
  $\omega = 1$) implies $\norm{g(t)} \leq C \left ( \abs{t}^{1/2} + 1
  \right )$ holds for some $C < \infty$ and all $t \in \C$. By the
  Cauchy estimates, such $g$ is a constant function and $Q = 0$
  follows.

  Suppose $\mathcal R_Q = e^{i \phi}\Sector{\delta} \cup \{ 0 \}$ for
  $\delta > 0$ and $\phi \in [-\pi, \pi)$.  By claim
    \eqref{InclusionThmClaim4} of Theorem \ref{InclusionThm} we have
    $\mathop{conv}(\mathcal A_Q) = \C$. Now, if $\mathcal A_Q \subset
    \mathcal T_Q^{1/2}$ holds, then $Q = 0$ follows from the first
    claim.
\end{proof}

To get the main result of this section, we proceed to real operators
on partially ordered complex function spaces.  Let us assume that the
Banach space $X$ is a \emph{complex function space}, meaning that each
$x \in X$ is actually a function $x:\Omega \to \C$ where $\Omega$ is a
set of points.  We say that $x \in X$ is \emph{real} if $x(\omega) \in
\R$ for all $\omega \in \Omega$, and \emph{positive} if $x(\omega)
\geq 0$ for all $\omega \in \Omega$.  The conjugate and absolute value
of $x$ are defined as usual by $\overline{x}(\omega) :=
\overline{x(\omega)}$ and $\abs{x}(\omega) := \abs{x(\omega)}$ for all
$\omega \in \Omega$.  We require from $X$ that
\begin{enumerate}
\item $x \in X \Leftrightarrow \overline{x} \in X$ as well as $x \in
  X \Leftrightarrow \abs{x} \in X$;
 \item $\norm{\overline{x}} = \norm{\abs{x}} = \norm{x}$ for all $x
   \in X$; and
\item $\norm{x} \leq \norm{y}$ if $ 0 \leq x(\omega) \leq y(\omega)$
  for all $\omega \in \Omega$.
\end{enumerate}
It follows from these properties that each $x \in X$ has a
decomposition $x = x_{11} - x_{12} + i (x_{21} - x_{22})$ where all
$x_{j,k} \in X$ are positive and satisfy $\norm{x_{j,k}} \leq
\norm{x}$.

The conjugate of an operator $T \in \BLO(X)$ is defined by
$\overline{T}x := \overline{T \overline{x}}$ for $x \in X$, and the
operator is called \emph{[positive] real} if it maps [positive] real
vectors in $X$ to [positive] real vectors. If $-T$ is positive, then
$T$ is, of course, \emph{negative}. All products of real operators are
real, and the same holds for positive operators.  Squares of negative
operators are positive but the same does not generally hold for
general real operators: e.g., $T^2 = -I$ if $T = \sbm{0 & 1 \\ -1 &
  0}$.  We always have $\overline{T x}=\overline{T}\overline{x}$,
$\overline{\overline{T}} = T$ and $\norm{\overline{T}} = \norm{T}$ but
$\norm{\overline{T}T} \neq \norm{T}^2$ for positive real $T = \sbm{0 &
  1 \\ 0 & 0}$.  The operator conjugation is generally not an
involution in $\BLO(X)$ since $\overline{T
  S}=\overline{T}\,\overline{S} \neq \overline{S} \, \overline{T}$
unless $[T,S] = 0$.  If $Q\in \BLO(X)$ is a quasinilpotent real
operator, we have clearly $\overline{T(z)} = T(\overline{z})$,
$\overline{T_z} = T_{\overline{z}}$, $\overline{e^{tz Q}} = e^{t
  \overline{z} Q}$ for $t \in \R$ which implies that all the
microspectral sets $\mathcal A_Q$, $\mathcal B_Q$, $\mathcal G_Q$,
$\mathcal R_Q$, $\mathcal T_Q$, and $\mathcal T_Q^{1/2}$ are conjugate
symmetric. It is this symmetry that makes it possible to use
Phragmen--Lindel\"of theorem for excluding points in
\eqref{InclusionEquivalence}.
\begin{theorem}
  \label{PhragmenLindelofMainThm}
  Let $Q \in \BLO(X)$, $Q \neq 0$, be a quasinilpotent operator such
  that $zQ$ is a positive real operator (in the sense described above)
  for some $z \in \C$, $z \neq 0$.  If $\mathcal G_Q \neq \{ 0 \}$,
  then $\mathcal T_Q = \mathcal B_Q^{\circ} \cup \{ 0 \} \subset
  \mathcal B_Q$ and $\mathcal T_Q^{1/2} \subset \mathcal A_Q$.
\end{theorem}
\noindent The inclusion $\mathcal T_Q^{1/2} \subset \mathcal A_Q$ is
strict if $\mathcal R_Q \neq \{ 0 \}$ since  $\mathop{conv}(\mathcal G_Q \cup
\mathcal T_Q^{1/2}) \subset e^{i \phi} \clos{\C}_+$ for some half
plane $e^{i \phi} \clos{\C}_+ \subset \mathcal A_Q$; see Proposition
\ref{PhragmenLindelofMainProp}.
If $\mathcal B_Q = \mathcal G_Q = \mathcal R_Q = \{ 0 \}$, then we
only know (by the same proposition) that $\mathcal T_Q^{1/2}$ is
contained in some closed half plane whose boundary contains the
origin.
\begin{proof}
  We may assume without loss of generality that $Q$ itself is a
  negative real operator (i.e., $z = - 1$ in the statement of this
  theorem), and then all the sets $\mathcal A_Q$, $\mathcal B_Q$,
  $\mathcal G_Q$, $\mathcal R_Q$, $\mathcal T_Q$, and $\mathcal
  T_Q^{1/2}$ are conjugate symmetric (by the real operator property),
  and $\mathcal G_Q \subset \C_+ \cup \{ 0 \}$ (by negativity).  We
  divide the proof into two cases depending whether $\mathcal R_Q \neq
  \{ 0 \}$ or $\mathcal R_Q = \{ 0 \}$ (but $\mathcal G_Q \neq \{ 0
  \}$).  By claim \eqref{InclusionThmClaim2} of Theorem
  \ref{InclusionThm}, Theorem \ref{InteriorBQoThm}, and Proposition
  \ref{InterioInclusionProp}, to prove $\mathcal T_Q = \mathcal
  B_Q^{\circ} \cup \{ 0 \}$ it is enough to show that $\mathcal T_Q
  \subset \mathcal B_Q$.  By claims \eqref{InclusionThmClaim2},
  \eqref{InclusionThmClaim2b}, and \eqref{InclusionThmClaim3} of
  Theorem \ref{InclusionThm}, it is enough to just show that $\mathcal
  T^{1/2}_Q \subset \mathcal A_Q$.

  Case $\mathcal R_Q \neq \{ 0 \}$: Now $\mathcal R_Q =
  \Sector{\theta'} \cup \{ 0 \}$ for the Ritt angle $\theta' > 0$ and
  $\Sector{\pi/2 + \theta'} \subset \mathcal A_Q$ (see claim
  \eqref{InclusionThmClaim4} of Theorem \ref{InclusionThm}).  We
  conclude from this by simple geometry that
  \begin{equation*}
    \mathcal A_Q^c  \setminus (- \mathcal G_Q )
    \subset e^{\tfrac{3 \pi i}{4}} \clos{\Sector{\pi/4 -\theta'}} \cup
    e^{-\tfrac{3 \pi i}{4}} \clos{\Sector{\pi/4 -\theta'}}. 
  \end{equation*}
  For contradiction, suppose that $\mathcal T^{1/2}_Q \subset \mathcal
  A_Q$ does not hold. Then by \eqref{InclusionEquivalence} and
  conjugate symmetry of $\mathcal T^{1/2}_Q$ there are points $re^{i
    \phi}, re^{-i \phi} \in \mathcal T_Q$ where $r > 0$ and $\phi \in
  [\pi/2 + \theta', \pi - \theta']$. By claim
  \eqref{TauberianGrowthboundsClaim2} of Proposition
  \ref{TauberianGrowthbounds}, the entire function $g(t) = e^{tQ}$
  grows at most like a square root on the rays $e^{i \phi} \R_+$ and
  $e^{- i \phi} \R_+$.  Also $\overline{\R}_+ \subset \mathcal G_Q$ by
  Theorem \ref{SemigroupInteriorThm}, and $g$ is bounded on this
  ray. These three rays divide the whole $\C$ into three sectors whose
  central angles are strictly less than $\pi$.  Using the same
  argument as in the proof of Proposition
  \ref{PhragmenLindelofMainProp}, we conclude the contradiction $Q =
  0$. Thus $\mathcal T^{1/2}_Q \subset \mathcal A_Q$ follows as
  claimed.
  
  Case $\mathcal R_Q = \{ 0 \}$ but $\mathcal G_Q \neq \{ 0 \}$:
  Proposition \ref{TrivialitiesProp} and Theorem
  \ref{SemigroupInteriorThm} imply that $\mathcal G_Q = e^{i \phi}
  \clos{\R}_+$ for some $\phi \in [-\pi, \pi)$.  Hence $\mathcal
G_Q = \clos{\R}_+$ and $\C_+ \subset \mathcal A_Q$ because $Q$ is
negative real. Let us first prove that $\mathcal T_Q^{1/2} \subset
\clos{\C}_+$.

  For contradiction, suppose that $re^{i \phi} \in \mathcal T_Q^{1/2}$
  with $\phi \in (\pi/2,\pi)$; note that $\phi = \pi$ is excluded by
  Proposition \ref{GelfandHilleProp}. By conjugate symmetry, also
  $re^{-i \phi} \in \mathcal T_Q^{1/2}$. Now, the function $g(t) =
  e^{tQ}$ grows at most like a square root on the rays $e^{\pm i \phi}
  \R_+$by claim \eqref{TauberianGrowthboundsClaim1} of Proposition
  \ref{TauberianGrowthbounds}.  Since the function $g$ is bounded on
  $\R_+ \subset \mathcal G_Q$, it follows that $Q = 0$ just like in
  the first part of this proof. This contradiction proves $\mathcal
  T_Q^{1/2} \subset \clos{\C}_+ \subset \overline{\mathcal A_Q}$.

  To conclude the proof, it remains to show that $\mathcal T_Q^{1/2}
  \cup i \R = \{ 0 \}$. Suppose not, meaning that we should have $\pm
  i y \in \mathcal T_Q^{1/2}$ for $y > 0$ by the conjugate symmetry of
  $\mathcal T_Q^{1/2}$.  Clearly there exists $K' > 0$ such that
  $\norm{(1 \pm i y Q)^k} \leq 1 + K' k^{1/2} $ for all $k \geq 0$,
  and thus $\norm{(1 + y^2 Q^2)^k} =\norm{(1 + i y Q)^k (1 - i y Q)^k}
  \leq 1 + (2K'+1) k \leq K (k + 1)$ for all $k \geq 0$ with $K :=
  2K'+1$.
  
  Since $Q$ is a \emph{negative} real operator, the operator $Q' := y^2 Q^2$
  is positive real.  If $x \in X$ is a positive vector, then $\abs{(1
    + e^{i \theta} Q')^k x} (\omega) \leq [(1 + Q')^k x] (\omega)$ for
  all $\omega \in \Omega$ and $\theta \in [-\pi, \pi)$, and thus
    $\norm{(1 + e^{i \theta} Q')^k x} \leq \norm{(1 + Q')^k
      x}$. Presenting any $x \in X$ in terms of four positive $x_{j,k}
    \in X$ satisfying $\norm{x_{j,k}} \leq \norm{x}$ for $j,k = 1,2$,
    we obtain
    \begin{equation*}
      \norm{(1 + e^{i \theta} Q')^k} \leq 4 \norm{(1 + Q')^k} \leq 4 K(k +
      1) \text{ for all } \theta \in [\pi, \pi) \text{ and } k \geq 0.
    \end{equation*}
    Using this estimate on power series of the exponential functions
    gives $\norm{e^{t (1 + e^{i \theta} Q')}} \leq 4 K (1 + t)e^t$ for
    $t \geq 0$, and hence $\norm{e^{t Q'}} \leq 4 K (1 + \abs{t})$ for
    all $t \in \C$. Since $\frac{1}{2 \pi i} \int{e^{t Q'} t^{-k + 1}
      \, dt} = Q'^k/k!$ (where the integration is around any circle $r
    \T$ for $r > 0$), we obtain the estimate $\norm{Q'^k}/k! = 4 K (r
    + 1) r^{-k+2}$ for all $r > 0$. Putting $k = 4$ and letting $r \to
    \infty$ gives $Q^8 = 0$.  We conclude by Proposition
    \ref{AlgebraicOperatorProp} that $\mathcal T_Q^{1/2} = \{ 0 \}$
    which is a contradiction against $y > 0$.
\end{proof}
Even though we require from $X$ that its elements are functions
defined on all of $\Omega$, it is not difficult to extend the
definition of $X$ to spaces like $L^p(\Omega;\C)$, $1 \leq p \leq
\infty$ where $\Omega \subset R^n$ is equipped with the Lebesgue
measure. Of course, then the partial ordering structure is defined
only almost everywhere. This leads to Banach lattices but we leave
such generalizations to the reader. Another way of doing this is to
consider first the vector space $C(\Omega;\C)$ equipped with the
$L^p$-norm and then proceeding by a density argument. With this
extension, we see that the results of this section can be applied to
Riemann--Liouville operators $V^{\alpha}$ (that are positive real) as
introduced in Remark \ref{VolterraRemark}.

\begin{proposition}
  \label{PhragmenLindelofAlternativeProp}
  Let $Q \in \BLO(X)$ be a quasinilpotent operator such
  that $zQ$ is a positive real operator for some $z \in \C$, $z \neq
  0$.  If $\mathcal R_Q \neq \{ 0 \}$, then at least one of the
  following holds: {\rm(i)} $\mathcal T_Q^{1/2} \subset \clos{\mathcal
    G_Q}$, or {\rm(ii)} $\mathcal B_Q = \mathcal G_Q$.
\end{proposition}
\noindent By Corollary \ref{InclusionThmCor} we have either
$\mathcal T_Q^{1/2} \subset \clos{\mathcal G_Q}$ or $\mathcal G_Q
\subset \mathcal T_Q^{1/2}$ for such operators.
\begin{proof}
Without loss of generality assume that $Q$ is a negative real
operator.  Suppose that $\mathcal R_Q \neq \{ 0 \}$ but $\mathcal
T_Q^{1/2} \subset \clos{\mathcal G_Q}$ does not hold.  Then there
exists $z \in \C$ such that $z, \overline{z} \in \mathcal T_Q^{1/2}
\setminus \clos{\mathcal G_Q}$.  By interpolating between $z$,
$\overline{z}$, and the points of the sector $\mathcal R_Q$ using
Proposition \ref{RittWeakTauberianInterpProp}, we conclude that
$\partial \mathcal G_Q \subset \mathcal T_Q^{1/2}$; thus
$\clos{\mathcal G_Q} \subset \mathcal T_Q^{1/2}$ since $\mathcal
G_Q^{\circ} = \mathcal B_Q^{\circ} \subset \mathcal T_Q \subset
\mathcal T_Q^{1/2}$ by Proposition \ref{InterioInclusionProp} and
Theorem \ref{OpenInteriorCLosureThm}. Now $\mathcal B_Q = \mathcal
G_Q$ follows from claims \eqref{InclusionThmClaim1} and
\eqref{InclusionThmClaim3} of Theorem \ref{InclusionThm}.
\end{proof}

\subsection{Quasinilpotent operators in Schatten classes}

We proceed to compact quasinilpotent operators on a separable Hilbert
space $X$. The \emph{approximation numbers} are defined by
$\sigma_j(Q) := \inf_{\mathop{rank}{F} = j}{\norm{Q - F}}$ for $j = 0,
1, \ldots$. Clearly $\sigma_0(Q) = \norm{Q}$ and $\lim_{j \to
  \infty}{\sigma_j(Q)} = 0$ is equivalent with the compactness of $Q$.
For $p \in (0,\infty)$, the \emph{Schatten $p$-class} $S_p(X)$ is
defined by those $Q$ for which the norm
\begin{equation*}
  \norm{Q}_{S_p(X)}^p := \sum_{j \geq 0}{\sigma_j(Q)^p}
\end{equation*}
is finite. It is easy to see that $S_p(X)\subset S_{p'}(X)$ for $p <
p'$, $S_p(X)$ is a Banach space under this norm, and it has also the
ideal property $B Q \in S_p(X)$ whenever $B \in \BLO(X)$ and $Q \in
S_p(X)$. For $p = 1$ the space $S_p(X)$ is called the \emph{trace
  class}, and we have the \emph{Hilbert--Schmidt} operators for $p =
2$. The Fredholm resolvent $r(s) := (1 - sQ)^{-1}$ of quasinilpotent
$Q$ is an entire function of finite exponential type $1/p$ if $Q \in
S_p(X)$.  Indeed, we have the generalized Carleman inequality for $m
\in \N$ and $p \in (m,m + 1]$:
\begin{equation}
  \mathlabel{CarlemanInEq}
  \norm{(I - s Q)^{-1}} \leq me^{3\norm{Q}_{S_p(X)} \abs{s}^p } \quad \text{ for all } \quad s \in \C;
\end{equation}
see, e.g., \cite[Theorem 5.4]{JM:PIC}, \cite[Theorem 5.8.9]{ON:CIL},
and \cite[p. 1088--1119]{D-S:LOII}.  Using \eqref{CarlemanInEq} makes
it possible to apply Phragmen--Lindel\"of theorem on Fredholm
resolvent functions:
\begin{proposition}
\label{SchattenRittAngleProp}
  Let $Q \in \BLO(X)$, $Q \neq 0$, be a quasinilpotent operator such
  that $Q \in S_p(X)$ for some $p > 1$.  Then the Ritt angle of $Q$ satisfies
$\theta \leq \frac{\pi}{2} \left (1 - 1/p \right )$.
\end{proposition}
\begin{proof}
  Denoting the Ritt angle by $\theta$, we conclude from claim
  \eqref{InclusionThmClaim4} of Theorem \ref{InclusionThm} that
  $\mathcal A_Q^c$ is contained in a closed sector $e^{i \phi}
  \clos{\Sector{\pi - 2 \theta}}$ for some $\phi \in [-\pi,\pi)$.  We
    thus have three rays in $\mathcal A_Q$ that divide $\C$ into three
    closed sectors $\Sigma_i$, $i = 1,2,3$ so that the largest of
    their central angles $\alpha$ satisfies $\alpha = \pi - 2 \theta +
    \epsilon$ where $\epsilon > 0$ can be chosen arbitrarily small.

If $Q \in S_p(X)$ for $p > 1$, then \eqref{CarlemanInEq} holds and
Proposition \ref{PhragmenLindelofThm} can be applied to $(I - s
Q)^{-1}$ on each $\Sigma_i$ separately. If $\alpha < \pi / p$, we
conclude that $(I - s Q)^{-1}$ is bounded on all of $\C$, and the
contradiction $Q = 0$ follows from Liouville's theorem. Hence $\alpha
\geq \pi / p$, and the proof is completed by letting $\epsilon \to
0+$.
\end{proof}
\begin{corollary}
\label{SchattenNoRittCor}
  Let $Q \in \BLO(X)$, $Q \neq 0$, be a quasinilpotent operator such
  that $Q \in S_p(X)$ for all $p > 1$.  Then $\mathcal R_Q = \{0
  \}$.
\end{corollary}
Considering the Riemann--Liouville operators $Q = -V^{\alpha}$ on
$L^2(0,1)$ for $\alpha > 0$ in Remark \ref{VolterraRemark}, their
approximation number asymptotics are known to be $\sigma_j(V^{\alpha})
\approx (\pi j)^{- \alpha}$; see \cite{T-G:ASVVIO}. Hence $V^{\alpha}
\in S_p(L^2(0,1))$ if and only if $p > 1/\alpha$. In particular,
Corollary \ref{SchattenNoRittCor} implies $\mathcal R_{Q} = \{ 0
\}$ for $Q = -V^1$, and the result is sharp by \cite[p.
  137]{YL:SPSOSRRC}.


By Theorem \ref{InteriorBQoThm}, the Ritt set cannot \emph{strictly}
contain any of the half planes $e^{i \phi} \C_+ \cup \{ 0 \}$ for
$\phi \in [-\pi,\pi)$ if $Q \neq 0$.  Even the restricting case is
  also impossible unless $(I - s Q)^{-1}$ is quite pathological:
\begin{proposition}
\label{HalfPlaneRittProp}
  Let $Q \in \BLO(X)$ be a quasinilpotent operator such that $\mathcal
  R_Q = e^{i \phi} \C_+ \cup \{ 0 \}$ for some $\phi \in [-\pi,\pi)$.
    Then the Fredholm resolvent $r(s) = (I - s Q)^{-1}$ is not an
    entire function of exponential type in the sense of
    \eqref{PhragmenLindelofThmEq1}. In particular, $Q \notin S_p(X)$
    for all $p > 0$.
\end{proposition}
\begin{proof}
  If $Q$ is as assumed, then $\C \setminus e^{i \theta} \R_+ \subset
  \mathcal A_Q$ for some $\theta \in [-\pi, \pi)$ by claim
    \eqref{InclusionThmClaim4} of Theorem \ref{InclusionThm}. If
    $r$ is of some bounded type, we can divide the whole of
    $\C$ into a finite number of sufficiently small sectors and use
    Proposition \ref{PhragmenLindelofThm} on each of them
    separately. Again, the boundedness of the entire function
    $r$ would follow, and hence $Q = 0$.  That $Q \notin
    S_p(X)$ follows from \eqref{CarlemanInEq} for $p > 1$, and from
    \cite[Theorem 3.8]{JM:PIC} for $p \in (0, 1]$.
\end{proof}

\section{\label{AbelSec} The Abel set $\mathcal A_Q$ revisited}

Let us start by stating that if $\mathcal R_Q \neq \{ 0 \}$, there are
sectors inside $\mathcal A_Q$ on which the Fredholm resolvent $(I -
sQ)^{-1}$ is uniformly bounded:
\begin{lemma}
\label{FredholmResolventBoundedOnSectorsLemma}
Assume that $Q \in \BLO(X)$, $Q \neq 0$, is a quasinilpotent operator
such that $\mathcal R_Q$ has a nonempty interior.  Define $E :=
\mathcal R_Q \C_+ \cap \T$, $E'' := \mathcal A_Q \cap \T$, and by $E'$
denote the (path connected) component of $E''$ that contains
$E$. Define the function $\Phi: E'' \to \R_+$ by setting
  \begin{equation} \mathlabel{RadialMaximalFunction}
    \Phi(e^{i \theta}) := \sup_{r \geq 0}{\norm{(I - r e^{i \theta}Q)^{-1}}}.
  \end{equation}
  Then $E$ and $E'$are circular intervals, $E \subset E' \subset
  E''$, and the following holds:
\begin{enumerate}
\item \label{FredholmResolventBoundedOnSectorsLemmaClaim1} For any
   closed, circular interval $K \subset E$ we have $\sup_{e^{i
      \theta} \in K}{\Phi(e^{i \theta})} < \infty$.
\item \label{FredholmResolventBoundedOnSectorsLemmaClaim2} We have
  $\sup_{e^{i \theta} \in E}{\Phi(e^{i \theta})} = \infty$.
\item \label{FredholmResolventBoundedOnSectorsLemmaClaim3} For any
  closed, circular interval $K \subset E'$, there is an open circular
  interval $W \subset K$such that $\sup_{e^{i \theta} \in
    W}{\Phi(e^{i \theta})} < \infty$.
\end{enumerate}
\end{lemma}
\noindent Indeed, we do not know whether $\mathcal A_Q \setminus \{ 0
\}$ is always a connected set.  Proposition \ref{AbelBoundaryProp}
shows that $E' = E$ for $Q$ that satisfy a compactness assumption.
\begin{proof}
  Claim \eqref{FredholmResolventBoundedOnSectorsLemmaClaim1}: Because
  $t \mapsto e^{t Q}$ is an analytic semigroup on $\mathcal R_Q$ that
  is bounded on each closed subsector $\clos{\Sigma}$ of $\mathcal
  R_Q$, the Fredholm resolvent $(I - sQ)^{-1}$ is uniformly bounded on
  corresponding closed subsectors $\clos{\Sigma} \clos{\C}_+$ by the
  Hille--Yoshida theorem.
  
  Claim \eqref{FredholmResolventBoundedOnSectorsLemmaClaim2}: If
  $\sup_{e^{i \theta} \in E}{\Phi(e^{i \theta})} = \sup_{s \in
    \mathcal R_Q \C_+}{\norm{(I - s Q)^{-1}}} < \infty$, we would
  conclude by Proposition \ref{RittSetProp} that the bounding ray(s)
  of $\mathcal G_Q$ in $\partial \mathcal G_Q$  would belong to $\mathcal R_Q$. This is impossible
  by Theorem \ref{OpenInteriorCLosureThm}.

  Claim \eqref{FredholmResolventBoundedOnSectorsLemmaClaim3}: The
  function $\Phi$ defined by \eqref{RadialMaximalFunction} is clearly
  nonnegative and lower semicontinuous.  Define the level sets for $k
  = 1, 2, \ldots$ by $E_k := \{ \theta \in E' : \Phi(e^{i \theta})
  \leq k \}$. By lower semicontinuity, all of these sets are closed,
  and clearly $E_k \subset E_j$ for $k < j$ as well as $E' = \cup_{j
    \geq 1}{E_j}$. Defining $K_j := K \cap E_j$ we get an increasing
  family of closed sets satisfying $K = \cup_{j \geq 1}{K_j}$. Because
  $K$ is a complete metric space, there exists $j \in \N$ such that
  $K_j$ has a non-empty open interior $K_j^\circ$ by Baire's category
  theorem.  Hence, there exists an open circular interval $W \subset
  K_j^\circ$ such that $ \Phi(e^{i \theta}) \leq j$ for all $e^{i
    \theta} \in W$.  This completes the proof.
\end{proof}

The operator $Q = - V^\alpha$ for $\alpha \in (0,1)$ provides us with
an example of an operator for which $\mathcal A_Q$ is not convex:
\begin{proposition}
  Let $Q \in \BLO(X)$ be quasinilpotent operator such that $\mathcal
  R_Q \neq \{0 \}$.  Then one of the following holds: {\rm(i)} $\mathcal
  A_Q$ is not convex; {\rm(ii)} the Fredholm resolvent $r(s) := (1 -
  sQ)^{-1}$ is not of any exponential type; or {\rm(iii)} $Q = 0$.
\end{proposition}
\begin{proof}
  Assume that $\mathcal A_Q$ is convex, and that $r$ is of some
  exponential type. We prove that $Q = 0$.
  
  Since $\mathcal R_Q$ has a nonempty interior, we see from claim
  \eqref{InclusionThmClaim4} of Theorem \ref{InclusionThm} that
  $\mathop{conv}(\mathcal A_Q) = \C = \mathcal A_Q$. Thus the entire
  function $r$ is bounded on all rays in the sense that $\sup_{r \geq
    0}{\norm{(I - r e^{i \theta}Q)^{-1}}} < \infty$ for all $\theta
  \in [-\pi, \pi)$.  Proposition \ref{PhragmenLindelofThm} implies
    that for any $e^{i \theta} \in \T$, there is an open sector
    $\Sigma(\theta)$ with $e^{i \theta} \in \Sigma(\theta)$ on which
    $r$ is uniformly bounded by some constant $C(\theta) < \infty$.
    The sets $\Sigma(\theta) \cap \T$ are an open cover for $\T$, and
    hence there is a finite sub-cover. From this it follows that $r$
    is uniformly bounded on all of $\C$, and $Q = 0$ by Liouville's
    theorem.
\end{proof}
We have $\mathcal R_Q \C_+ \subset \mathcal A_Q$ by claim
\eqref{InclusionThmClaim4} of Theorem \ref{InclusionThm} but we cannot
exclude in Lemma \ref{FredholmResolventBoundedOnSectorsLemma} the
possibility that $\mathcal A_Q$ could be substantially larger than
$\mathcal R_Q \C_+$. If the Fredholm resolvent is of exponential type
(e.g., if $Q \in S_p(X)$ for $p > 0$), then we know that at least a
part of $\partial \mathcal A_Q$ is where one would expect:
\begin{proposition}
\label{AbelBoundaryProp}
  Let $Q \in \BLO(X)$ be a quasinilpotent operator with $\mathcal G_Q
  \neq \{ 0 \}$ and that the Fredholm resolvent $r(s) = (I - s
  Q)^{-1}$ is an entire function of exponential type.  Then $\partial
  ( \mathcal G_Q \C_+) \subset \partial \mathcal A_Q$.
\end{proposition}
\begin{proof}
We prove this under the stronger assumption $\mathcal R_Q \neq \{ 0
\}$ in which case $\mathcal G_Q \C_+ = \mathcal R_Q \C_+$ holds
because $\C_+$ is open and $\mathcal R_Q = \mathcal G_Q^\circ \cup
\{ 0 \}$. Then, by claims
\eqref{FredholmResolventBoundedOnSectorsLemmaClaim1} and
\eqref{FredholmResolventBoundedOnSectorsLemmaClaim2} of Lemma
\ref{FredholmResolventBoundedOnSectorsLemma}, the open sector
$\mathcal R_Q \C_+ \setminus \{ 0 \}$ has the following properties:
    {\rm(i)} for each closed sector $\Sigma \subset \mathcal R_Q
    \C_+$, the Fredholm resolvent $r$ is bounded on $\Sigma$, and
      {\rm(ii)} there is no larger sector than $\mathcal R_Q \C_+
      \setminus \{ 0 \}$ having the same property.

Since $\mathcal R_Q$ is a sector (see Theorem \ref{InteriorBQoThm}),
we have $\partial ( \mathcal R_Q \C_+) = e^{i \phi_1} \clos{\R_+} \cup
e^{i \phi_2} \clos{\R_+}$ for some $\phi_1, \phi_2 \in [-\pi, \pi)$.
  Suppose, for contradiction, that $ e^{i \phi_1} \clos{\R_+} \not
  \subset \partial \mathcal A_Q$. Since always $\mathcal R_Q \C_+
  \subset \mathcal A_Q$, it is impossible that $e^{i \phi_1} \R_+
  \subset \mathcal A_Q^{c} \setminus \partial \mathcal A_Q $. Hence,
  $e^{i \phi_1} \R_+ \subset \mathcal A_Q^{\circ}$ and we have $e^{i
    \phi_1} \clos{\Sector{\theta}} \subset \mathcal A_Q^{\circ} \cup
  \{ 0 \} \subset \mathcal A_Q$ for all $\theta > 0$ small enough. We
  may conclude that $r$ is bounded in $e^{i \phi_1}
  \clos{\Sector{\theta}}$ with sufficiently small $\theta > 0$ that is
  compatible with the exponential type of $r$ so that Proposition
  \ref{PhragmenLindelofThm} can be used.  Now, the set $\Sigma' :=
  e^{i \phi_1} \Sector{\theta} \cup \mathcal R_Q \C_+$ is strictly
  larger than $\mathcal R_Q \C_+ \setminus \{ 0 \}$ but it still
  satisfies condition {\rm(i)} given above.
\end{proof}

\section{\label{HilleYosidaSec} Miscellaneous observations}


We start by giving three results concerning the Ritt set $\mathcal
R_Q$.  We consider the case where $T_z$ in \eqref{OperatorFamilies} is
not only (uniformly) power-bounded as in
Proposition~\ref{HilleYosidaProp} (related to Hille--Yoshida semigroup
generator theorem) but more strongly, a Ritt operator characterized by
the resolvent estimate
\begin{equation*}
 \sup_{\xi \in 1 + \Sector{\pi /2 + \delta} } 
 {\abs{\xi - 1} \cdot \norm{\left ( \xi - T_z \right )^{-1}}} 
 < \infty.
\end{equation*}
For all $z \in \C$, define
\begin{equation} \label{HilleYosidaSecEq1}
    \quad Q_z := T_z - I  = zQ\left (1 - zQ \right )^{-1}.
\end{equation}


\begin{proposition} \label{RittHilleYosidaProp}
  Let $Q \in \BLO(X)$ be a quasinilpotent operator.  Then $T_z$ in
  \eqref{OperatorFamilies}, \eqref{HilleYosidaSecEq1} is a Ritt
  operator if and only if $z \in \mathcal R_Q$.
\end{proposition}
\noindent In other words, $z \in \mathcal R_{Q (I - zQ)^{-1}}$ if and
only if $1 \in \mathcal R_{Q_z}$ if and only if $z \in \mathcal
R_{Q}$.  This does not mean that $\mathcal R_{Q} = \mathcal R_{Q (I -
  zQ)^{-1}}$ for all $z$.
\begin{proof}
 For all $\xi \neq 1$ we have  the identity
\begin{equation*}
  \left ( \xi - 1 \right )\left ( \xi - T_z \right )^{-1}
= \left (1 - zQ \right ) \left ( 1 -  \frac{\xi z}{\xi - 1}Q \right )^{-1}.
\end{equation*}
Denoting $s = \xi/(\xi - 1)$ it is easy to see that $\xi \in 1 + \Sector{\pi /2
  + \delta}$ is equivalent with $s \in 1 + \Sector{\pi /2 + \delta}$.
Since $\rho(Q) = 0$ we have
\begin{equation*}
  \sup_{\xi \in 1 + \Sector{\pi /2 + \delta}} 
    {\norm{\left ( \xi - 1 \right )\left ( \xi - T_z \right )^{-1}}} < \infty \Leftrightarrow
  \sup_{s \in 1 + \Sector{\pi /2 + \delta}}  {\norm{\left ( 1 -  s zQ \right )^{-1} }} < \infty
\end{equation*}
for any $\delta > 0$.  Recall that $z \in \mathcal R_Q$ if and only if
$\sup_{\Re{s} > -1/2 } {\norm{\left ( 1 - s zQ \right )^{-1} }} <
\infty.$ Since the set $\{\Re{s} \geq -1/2 \} \setminus \left (1 +
  \Sector{\pi /2 + \delta} \right )$ is a closed triangle, and the mapping
$s \mapsto \left ( 1 - s zQ \right )^{-1}$ is continous for all
$z$, we conclude that
\begin{equation*}
  \sup_{s \in 1 + \Sector{\pi /2 + \delta}}  {\norm{\left ( 1 -  s zQ \right )^{-1} }} < \infty
  \Leftrightarrow 
  \sup_{\Re{s} > - 1/2}  {\norm{\left ( 1 -  s zQ \right )^{-1} }} < \infty.
\end{equation*}
This proves the claim.
\end{proof}

It is clear that $\sup_{k \in \N}{(k + 1) \norm{(I - T_z) T_z^k}} <
\infty$ for all $z \in \mathcal R_Q$ by \cite[Proposition
1.1]{M-N-Y:OTCFBLO} and Proposition \ref{RittHilleYosidaProp}. A similar
but weaker conclusion can be given in the larger set $\mathcal G_Q$:
\begin{proposition} \label{WeakerTauberianProp}
  Let $Q \in \BLO(X)$ be a quasinilpotent operator.  The operator
  $T_z$ in \eqref{OperatorFamilies} satisfies $\sup_{k \in \N}{\sqrt{k
      + 1} \norm{(I - T_z) T_z^k}} < \infty$ for all $z \in \mathcal
  G_Q$.
\end{proposition}
\begin{proof}
  Let $z \in \mathcal G_Q$ be arbitrary. Define for $\alpha \in (0,1)$
  the power-bounded operators $R_{\alpha,z} := (1 - \alpha) + \alpha
  T_z = I + \alpha Q_z$; see Proposition \ref{HilleYosidaProp} above.
  Since $I - R_{\alpha,z} = - \alpha z Q \left (1 - zQ \right )^{-1}$,
  we get by a straightforward computation
  \begin{equation} \label{WeakerTauberianPropEq1}
    - \frac{1}{\alpha} \left ( I - R_{\alpha,z} \right ) R_{\alpha,z}^k 
    \cdot \left (1 - (1 - \alpha) zQ \right )^{-k}
    = z Q \left (1 - zQ \right )^{-k-1}.
  \end{equation}
  The Hille--Yosida theorem implies that $\sup_{k \in \N}\norm{\left
      (1 - (1 - \alpha) zQ \right )^{-k}} < \infty$ because $1 -
  \alpha > 0$ and $z \in \mathcal G_Q$. Moreover, it follows from
  \cite[Theorem 4.5.3]{ON:CIL} that 
  \begin{equation} \label{WeakerTauberianPropEq2}
    \sup_{k \in \N}{\sqrt{k + 1}
      \norm{\left ( I - R_{\alpha,z} \right ) R_{\alpha,z}^k}} < \infty.  
  \end{equation} Noting that $z Q \left (1 - zQ \right )^{-k-1} = (I - T_z)
  T_z^k$, the claim follows from \eqref{WeakerTauberianPropEq1}.
\end{proof}
Note that $T_z = 1 + zQ + z^2Q^2\left (1 - zQ \right )^{-1}$
approximates the operator $T(z) = 1 + zQ$ for $z \approx 0$. If,
instead of Proposition~\ref{WeakerTauberianProp}, we had $\mathcal G_Q
\subset \mathcal T_Q^{1/2}$, then the equality $\mathcal B_Q =
\mathcal G_Q$ would follow by Corollary \ref{InclusionThmCor}.  This
is the motivation for Propositions \ref{RittHilleYosidaProp} and
\ref{WeakerTauberianProp}.

We complete this section by showing that the boundedness of the
Fredholm resolvent in small sectors implies practically nothing on the
the semigroups:
\begin{proposition} \label{AbelianImpliesNothingProp}
  There exists a quasinilpotent operator $Q \in \BLO(X)$ such that
  $\sup_{s \in \Sector{\delta}}{\norm{(I - sQ)^{-1}}} < \infty$ with some $0
  < \delta < \pi/2$ (hence, $\Sector{\delta} \subset \mathcal A_Q$) but the
  estimate
    $$\norm{e^{t Q}} \leq M_k t^k \text{ for all } t \geq 0$$
  does not hold for any $k \geq 1$ and $M_k < \infty$.
\end{proposition}
\noindent The operator $Q = -V^{1}$ in Remark
\ref{VolterraRemark} is sectorial so that $\sup_{s \in \Sector{\pi/2 -
    \eta}}{\norm{(I - s Q)^{-1}}} < \infty$ for all $\eta \in (0,
\pi/2)$; see, e.g., \cite{YL:SPSOSRRC}.
\begin{proof}
  By $Q'$ denote any quasinilpotent operator with $\mathcal R_{Q'}
  \neq \{0 \}$ (see Remark \ref{VolterraRemark}).  Such $Q'$ is
  never nilpotent by Proposition \ref{AlgebraicOperatorProp}, and
  without loss of generality we may assume that $\mathcal R_{Q'} =
  \Sector{2 \delta} \cup \{ 0 \}$ for $\delta > 0$.  Then by claim
  \eqref{InclusionThmClaim4} of Theorem \ref{InclusionThm} we have
  $e^{\pm i(\pi/2 + \delta)} \Sector{\delta} \subset \mathcal A_{Q'}$, and
  also $\sup_{s' \in e^{\pm i(\pi/2 + \delta)} \Sector{\delta}} {\norm{(I -
      sQ')^{-1}}} < \infty$ as can be seen by Proposition
  \ref{NevanlinnaLubichProp}. Both the operators $Q^{\pm} := e^{\pm i
    (\pi/2 + \delta)} Q'$ satisfy now the conditions of this
  proposition.

For contradiction, suppose that $\norm{e^{t Q^{\pm}}} \leq M_k t^k$
for some $k, M_k$ and all $t \geq 0$. Then we would have $\norm{e^{t
    Q'}} \leq M_k \abs{t}^k$ on the rays $e^{\pm i (\pi/2 + \delta)}
\R_+$, and also $\norm{e^{t Q'}} \leq M < \infty$ on
$\R_+$. Proposition \ref{PhragmenLindelofAlternativeProp} implies that
$\norm{e^{t Q'}} \leq M_k \abs{t}^k$ for all $t \in \C$, and hence
$e^{t Q'}$ is a polynomial. This is possible only if $Q$ is nilpotent
which is impossible by Proposition \ref{AlgebraicOperatorProp}.
\end{proof}

\section{Conclusions}

These conclusions concern the open problems that remain.

\subsection{\label{OpenProb1} Do the sets $\mathcal B_Q$ and 
$\mathcal K_Q^\infty$ consist of full rays?} In other words, do we
have $\mathcal B_Q \R_+ = \mathcal B_Q$?  We know that $\mathcal B_Q$
is convex.  Since the open interior $\mathcal B_Q^{\circ}$ is a
sector, it has the full ray property.  Thus, the question can be
rephrased whether $\left ( \partial \mathcal B_Q \cap \mathcal B_Q
\right ) \R_+ \subset \mathcal B_Q$. The same questions can be asked
about the set $\mathcal K_Q^\infty$, too.

\subsection{\label{OpenProb2} Do we always have $\mathcal B_Q = \mathcal G_Q$?} 
If this equality holds, then Problem \ref{OpenProb1} is clearly
resolved in positive. Since the convex sets $\mathcal B_Q$ and
$\mathcal G_Q$ have the same interior, the question is whether
$\partial \mathcal B_Q \cap \mathcal B_Q = \partial \mathcal G_Q \cap
\mathcal G_Q$. See Corollary \ref{InclusionThmCor} and Proposition
\ref{UniformlyBPonRayProp}.

\subsection{\label{OpenProb4} Is it possible to have $\mathcal B_Q \neq \mathcal R_Q$ or even $\mathcal G_Q \neq \mathcal
  R_Q$ when $\mathcal R_Q \neq \{ 0 \}$?} In other words, is one of
the sets $\partial \mathcal B_Q \cap \mathcal B_Q$ and $\partial
\mathcal G_Q \cap \mathcal G_Q$, or both, always empty for such $Q$?
If they are, then Problem \ref{OpenProb2} is resolved in positive for
operators with $\mathcal R_Q \neq \{ 0 \}$.

\subsection{\label{OpenProb5} Does the set $\mathcal K_Q$ consist of full rays, or is it convex?}
 It would follow from convexity that Theorem
 \ref{OpenInteriorCLosureThm} could be completed with
 $\overline{\mathcal R_Q} = \overline{\mathcal K_Q}$.

\subsection{\label{OpenProb6} Do we always have $\mathcal G_Q \cap \partial \mathcal T_Q  \subset \mathcal T_Q^{1/2}$?}
This inclusion has a flavour of a trace theorem, and there is certain
ring of truth in it. If $\partial \mathcal T_Q \subset \mathcal
T_Q^{1/2}$ holds for some $Q$ with $\mathcal R_Q \neq \{ 0 \}$, then
$\partial \mathcal G_Q \subset \mathcal T_Q^{1/2}$ since $\partial
\mathcal G_Q = \partial \mathcal B_Q = \partial \mathcal B_Q^{\circ}$
and $\mathcal A_Q \cap \mathcal T_Q = \mathcal R_Q = \partial \mathcal
B_Q^{\circ} \cup \{ 0 \}$. This would again resolve Problem
\ref{OpenProb2} in positive for operators with $\mathcal R_Q \neq \{ 0
\}$.

\subsection{\label{OpenProb3} Do we have the inclusion $\mathcal T_Q \subset \mathcal B_Q$ 
or even $\mathcal T_Q^{1/2} \subset \mathcal A_Q$ without the extra
assumptions of Section \ref{LindelofSec}?} This is a particularly deep
question, and Theorem \ref{PhragmenLindelofMainThm} gives hints what
kind of counter examples could work. Note that there is a counter
example \cite[Theorem 3.3]{K-M-O-T:PBORNE} 
for the same question without the quasinilpotency assumption.

\vspace{0.5cm}

All of these open problems seem to be resistent to the techniques we
have presented in this work. We probably need fresh, new ideas now,
and trying to produce counter examples seems like a reasonable next
step.
 
\providecommand{\bysame}{\leavevmode\hbox to3em{\hrulefill}\thinspace}
\providecommand{\MR}{\relax\ifhmode\unskip\space\fi MR }
\providecommand{\MRhref}[2]{%
  \href{http://www.ams.org/mathscinet-getitem?mr=#1}{#2}
}
\providecommand{\href}[2]{#2}


\begin{thebibliography}{KMSOT04}

\bibitem[ABHN01]{A-B-H-N:VVLTCP}
W.~Arendt, C.~Batty, M.~Hieber, and R.~Neubrander, \emph{Vector-valued
  {L}aplace transforms and {C}auchy problems}, Birkh\"auser, 2001.

\bibitem[Ber83]{MB:IPSAB}
M.~Berkani, \emph{In\'egalit\'es et propri\'et\'es spectrales dans les
  alg\'ebres de {B}anach}, Th\`ese de 3\'eme cycle, University of Bordeaux {I},
  1983.

\bibitem[DS63]{D-S:LOII}
N.~Dunford and J.~Schwartz, \emph{Linear operators; part ii: Spectral theory},
  Interscience Publishers, Inc.~(J.~Wiley \& Sons), New York, 1963.

\bibitem[Dun08a]{ND:OIFPO}
N.~Dungey, \emph{On an integral of fractional power operators}, Colloquium
  Mathematicum \textbf{117} (2008), no.~2, 157--164.

\bibitem[Dun08b]{ND:TRRCPBO}
\bysame, \emph{On time regularity and related conditions for power bounded
  operators}, Proc. London Math. Soc. \textbf{97} (2008), no.~1, 97--116.

\bibitem[Est83]{JE:QRCIPCBA}
J.~Esterle, \emph{Quasimultipliers, representations of ${H}\sp{\infty }$, and
  the closed ideal problem for commutative {B}anach algebras}, Radical Banach
  algebras and automatic continuity (Long Beach, Calif., 1981) (Berlin),
  Lecture Notes in Mathematics, vol. 975, Springer Verlag, 1983, pp.~66--162.

\bibitem[FNRR90]{F-N-R-R:WRLOII}
C.~K. Fong, E.~A. Nordgren, H.~Radjavi, and P.~Rosenthal, \emph{Weak resolvents
  of linear operators, {II}}, Indiana University Mathematics Journal
  \textbf{39} (1990), no.~1, 67--83.

\bibitem[FW73]{F-W:CPSCMO}
S.~R. Foguel and B.~Weiss, \emph{On convex power series of a conservative
  {M}arkov operator}, Proceedings of the american mathematical society
  \textbf{38} (1973), no.~2, 325--330.

\bibitem[KMSOT04]{K-M-O-T:PBORNE}
N.~Kalton, S.~Montgomery-Smith, K.~Oleszkiewicz, and Y.~Tomilov,
  \emph{Power-bounded operators and related norm estimates}, J. London Math.
  Soc. (2) \textbf{70} (2004), no.~2, 463--478. \MR{MR2078905 (2005e:47020)}

\bibitem[Lyu01]{YL:SPSOSRRC}
Yu. Lyubich, \emph{The single-point spectrum operators satisfying {R}itt's
  resolvent condition}, Studia Mathematica \textbf{145} (2001), 135--142.

\bibitem[Lyu10]{YL:PWRCDCVO}
\bysame, \emph{The power boundedness and resolvent conditions for functions of
  the classical {V}olterra operator}, Studia Mathematica \textbf{196} (2010),
  no.~1, 41--63.

\bibitem[Mal96]{JM:PIC}
J.~Malinen, \emph{On the properties for iteration of a compact operator with
  unstructured perturbation}, Tech. Report A360, {H}elsinki {U}niversity of
  {T}echnology {I}nstitute of {M}athematics, 1996.

\bibitem[MNTY07]{M-N-Y:LBOXX}
J.~Malinen, O.~Nevanlinna, V.~Turunen, and Z.~Yuan, \emph{A lower bound for the
  differences of powers of linear operators}, Acta Mathematica Sinica
  \textbf{23} (2007), no.~4, 745--748.

\bibitem[MNY09]{M-N-Y:OTCFBLO}
J.~Malinen, O.~Nevanlinna, and Z.~Yuan, \emph{On the {Tauberian} condition for
  bounded linear operators}, Mathematical Proceedings of the Royal Irish
  Academy \textbf{109} (2009), 101--108.

\bibitem[Nev93]{ON:CIL}
O.~Nevanlinna, \emph{Convergence of iterations for linear equations}, Lectures
  in Mathematics ETH Z{\"u}rich, Birkh\"auser Verlag, Basel, Boston, Berlin,
  1993.

\bibitem[Nev97]{ON:GROPB}
O.~Nevanlinna, \emph{On the growth of the resolvent operators for power bounded
  operators}, Linear Operators (F.H.~Szafraniec J.~Janas and J.~Zem\'{a}nek,
  eds.), Banach Center Publications, vol.~38, Inst. Math., Polish Acad. Sci.,
  1997, pp.~247--264.

\bibitem[NRR87]{N-R-R:WRLO}
E.~A. Nordgren, H.~Radjavi, and P.~Rosenthal, \emph{Weak resolvents of linear
  operators}, Indiana University Mathematics Journal \textbf{36} (1987), no.~4,
  913--934.

\bibitem[Pey69]{AP:LOS}
A.~Peyerimhoff, \emph{Lectures on summability}, vol. 107, Springer Verlag,
  Berlin, 1969.

\bibitem[TG96]{T-G:ASVVIO}
V.~K. Tuan and R.~Gorenflo, \emph{Asymptotics of singular values of {V}olterra
  integral operators}, Numerical Functional Analysis and Applications
  \textbf{17} (1996), 453--461.

\bibitem[Ton89]{YT:QIO}
Y.S. Tong, \emph{Quasinilpotent integral operators}, Acta Mathematica Sinica
  \textbf{32} (1989), 727--735.

\bibitem[Tse03]{DT:OPBCVOP}
D.~Tsedenbayar, \emph{On the power boundedness of certain {V}olterra operator
  pencils}, Studia Mathematica \textbf{156} (2003), no.~1, 59--66.

\bibitem[Whi87]{RW:SVCO}
R.~Whitley, \emph{The spectrum of a {V}olterra composition operator}, Integral
  Equations Operator Theory \textbf{10} (1987), 146--149.

\bibitem[Zem94]{JZ:OGHT}
J.~Zem\'anek, \emph{On the {G}elfand-{H}ille theorems}, Functional Analysis and
  Operator Theory (J.~Zem\'anek, ed.), vol.~30, {B}anach center publications,
  1994, pp.~369--385.

\end{thebibliography}

\end{document}